\documentclass[reqno, 11pt]{amsart}
\usepackage{amsmath, amssymb, amsthm,bm}
\usepackage[margin=1in]{geometry}
\usepackage[T1]{fontenc}
\usepackage[hidelinks]{hyperref}

\usepackage[shortlabels]{enumitem}
\usepackage{url}
\usepackage{tikz}
\usepackage{caption}
\usepackage[labelformat=simple]{subcaption}

\theoremstyle{plain}
\newtheorem{thm}{Theorem}
\newtheorem{lemma}[thm]{Lemma}
 
\newtheorem{prop}[thm]{Proposition}

\theoremstyle{definition}
\newtheorem{defn}[thm]{Definition}

\newtheorem{example}[thm]{Example}
\newtheorem{nonexample}[thm]{Non-example}

\theoremstyle{remark}
\newtheorem{rem}[thm]{Remark}

\numberwithin{thm}{section}

\newcommand{\abs}[1]{\left\lvert#1\right\rvert}
\newcommand{\norm}[1]{\left\lVert#1\right\rVert}

\newcommand{\ang}[1]{\left\langle #1 \right\rangle}
\newcommand{\floor}[1]{\left\lfloor #1 \right\rfloor}

\newcommand{\paren}[1]{\left( #1 \right)}

\newcommand{\wh}{\widehat}
\newcommand{\ol}{\overline}

\newcommand{\bx}{\bm{x}}

\newcommand{\cH}{\mathcal{H}}
\newcommand{\cD}{\mathcal{D}}
\newcommand{\cE}{\mathcal{E}}
\newcommand{\cP}{\mathcal{P}}
\newcommand{\cQ}{\mathcal{Q}}

\newcommand{\EEE}{\mathop{\mathbb{E}}}
\newcommand{\FF}{\mathbb{F}}

\newcommand{\ZZ}{\mathbb{Z}}

\newcommand{\C}{\mathbb C}
\newcommand{\E}{\mathbb E}
\newcommand{\F}{\mathbb F}

\DeclareMathOperator{\fnz}{fnz}
\DeclareMathOperator*{\spn}{span}
\DeclareMathOperator{\rank}{rank}
\DeclareMathOperator{\codim}{codim}

\title[Induced arithmetic removal]{Induced arithmetic removal: \\ complexity 1 patterns over finite fields}
\author{Jacob Fox}
\address{Stanford University, Stanford, CA 94305, USA}
\email{jacobfox@stanford.edu}
\author{Jonathan Tidor}
\author{Yufei Zhao}
\address{Massachusetts Institute of Technology, Cambridge, MA 02139, USA}
\email{\{jtidor,yufeiz\}@mit.edu}
\thanks{Fox was supported by a Packard Fellowship and by NSF grant DMS-1855635. Tidor was supported by NSF Graduate Research Fellowship Program DGE-1122374. Zhao was supported by NSF Award DMS-1764176, the MIT Solomon Buchsbaum Fund, and a Sloan Research Fellowship.}
\date{}

\begin{document}

\begin{abstract}
We prove an arithmetic analog of the induced graph removal lemma for complexity 1 patterns over finite fields. Informally speaking, we show that given a fixed collection of $r$-colored complexity 1 arithmetic patterns over $\mathbb F_q$, every coloring $\phi \colon \mathbb F_q^n \setminus\{0\} \to [r]$ with $o(1)$ density of every such pattern can be recolored on an $o(1)$-fraction of the space so that no such pattern remains.
\end{abstract}

\maketitle

\section{Introduction}
\label{sec:intro}

Removal lemmas play a central role in extremal and additive combinatorics (see the survey \cite{CF13}). The \emph{triangle removal lemma}, first proved by Ruzsa and Szemer\'edi \cite{RS78}, states that for every $\epsilon > 0$ there exists $\delta > 0$ such that every $n$-vertex graph with fewer than $\delta n^3$ triangles can be made triangle-free by deleting at most $\epsilon n^2$ edges. 

In computer science, there is much interest in removal lemmas due to the connection between removal lemmas and constant query complexity algorithms for property testing~\cite{RS96,GGR98}. For example, the triangle removal lemma implies that to distinguish a triangle-free graph from one that is $\epsilon$-far from triangle free, one simply needs to sample a constant number of triples of vertices (this constant depends on $\epsilon$, but not the size of the graph) and check them for triangles. 
All our discussions and results below can be interpreted using the language of property testing.

An arithmetic analog of the triangle removal lemma was first established by Green \cite{G05}. It states that for every $\epsilon > 0$ there exists $\delta > 0$ such that for every finite abelian group $(G,+)$ and every $A \subset G$ with at most $\delta |G|^2$ triples $(x,y,z) \in A^3$ satisfying $x+y=z$, there exists $A' \subseteq A$ with $|A \setminus A'| \le \epsilon |G|$ such that $A'$ contains no solution to $x+y=z$.

Green's arithmetic triangle removal lemma also has a corresponding property testing algorithm. This algorithm distinguishes between \emph{triangle-free} maps $f\colon G\to\{0,1\}$ (i.e., those with no $x,y\in G$ satisfying $f(x)=f(y)=f(x+y)=1$) and maps that are $\epsilon$-far from triangle free by sampling a constant number of pairs $(x,y)\in G^2$ and checking if any satisfy $f(x)=f(y)=f(x+y)=1$.

Generalizing the triangle removal lemma, there is a \emph{graph removal lemma} which states that for every fixed graph $H$ (or a finite collection of graphs), every $n$-vertex graph with $H$-density $o(1)$ can be made $H$-free by removing $o(n^2)$ edges \cite{ADLRY94, F95}.  Furthermore, there is an \emph{induced graph removal lemma} \cite{AFKS00} which states that every $n$-vertex graph with induced $H$-density $o(1)$ can be made induced $H$-free by adding or removing $o(n^2)$ edges. Both of these removal lemmas give constant query complexity algorithms for property testing of the corresponding properties.

The graph removal lemma was proved using Szemer\'edi's graph regularity lemma, an important tool in extremal combinatorics.  The arithmetic removal lemma was initially proved by Green via an arithmetic variant of the graph regularity lemma; an alternate proof was later given by Kr\'al'--Serra--Vena \cite{KSV09}, deducing the arithmetic removal lemma from the graph removal lemma. This technique was extended to deduce a removal lemma for general linear systems from the hypergraph removal lemma \cite{KSV12, S10}.

A natural common extension of the induced graph (or hypergraph) removal lemma and the arithmetic removal lemma would be an induced arithmetic removal lemma, which is not currently known (e.g., \cite[Conjecture 5.3]{S10}). In particular, the technique for deducing arithmetic removal from graph removal (from \cite{KSV09}) does not appear to extend to the induced setting. 

The goal of this paper is to establish an induced arithmetic removal lemma for arithmetic patterns of complexity 1 over finite fields. (See Section~\ref{ssec:complexity} for the definition of complexity.) Previous results \cite{BGS15} established this claim over $\FF_2$, though the proof requires specific Ramsey-theoretic results that fail for all other finite fields. We overcome this difficulty by introducing a novel recoloring technique that works over all base fields.

It remains an open problem to extend the induced arithmetic removal lemma to more general settings, both to higher complexity patterns and to general groups (even $\ZZ$ is open). For affine-invariant patterns (which are always translation-invariant), a special case of the general problem that avoids an important difficulty (no regularization near the origin; see discussion in Section~\ref{sec:summary}), the problem for bounded complexity patterns over finite fields has been solved \cite{BFHHL13}.

We begin with a discussion motivating the precise setup of our main theorem. In Section~\ref{sec:summary} we review the background on the regularity proof of the graph removal lemma and its relationship to our proof, and we describe our proof strategy. In Section~\ref{sec:main-thm} we prove our main theorem, modulo two key tools, proved in Sections~\ref{sec:rado} and \ref{sec:regularity}. In Section~\ref{sec:extensions} we give two extensions of our main theorem, to an infinite set of patterns and to inhomogeneous patterns.

\subsection*{Acknowledgements}
The authors are grateful to Noga Alon, Freddie Manners, and Tom Sanders for helpful discussions.

\subsection{Colored patterns}
Induced subgraphs can be equivalently rephrased as colored patterns in a 2-edge-colored clique, where the two colors correspond to edges and non-edges. We consider the natural generalization allowing an arbitrary fixed number of edge colors. 

The proof of the induced graph removal lemma gives the following result about colored patterns. Given $r$-edge-colored cliques $G$ and $H$, the \emph{$H$-density} in $G$ is defined to be the fraction of $|H|$-vertex subgraphs of $G$ that have edge-coloring is isomorphic to $H$. We say that $G$ is \emph{$H$-free} if no such subgraph exists.

\begin{thm}[Induced graph removal]
	Fix a finite collection $\cH$ of $r$-edge-colored cliques. For every $\epsilon > 0$ there exists $\delta >0$ so that if an $r$-edge-coloring of $K_n$ has $H$-density at most $\delta$ for every $H \in \cH$, then one can recolor at most $\epsilon n^2$ edges of $K_n$ (using the original $r$ colors) so that the resulting coloring is $H$-free for every $H \in \cH$.
\end{thm}

An analogous statement for hypergraphs is also known to hold \cite{RS07}. In addition, there is even a result allowing an infinite collection of subgraphs $\cH$ (with a slightly modified statement); see Section~\ref{ssec:infinite-removal}.

We would like to obtain a similar statement for colored arithmetic patterns. However, there must be some caveats, as illustrated by the following two non-examples.

\begin{nonexample}
	Consider the 2-coloring of $\ZZ/N\ZZ$ with $1, \dots, \floor{N/2}$ colored red and the remaining elements colored blue, and the pattern $x = y+1$ where $x$ is blue and $y$ is red. There is only one instance of this pattern in the given coloring, but it is impossible to recolor fewer than $\floor{N/2}$ elements to remove this pattern. For an example with a homogeneous equation, a similar situation occurs for $x=2y$ as long as 2 has large multiplicative order in $(\ZZ/N\ZZ)^\times$.
\end{nonexample}

The above non-example does not worry us too much since it has infinite complexity and is not over a fixed finite field. Nevertheless, it reminds us that we need to be cautious when formulating a more general conjecture about induced arithmetic removal.

The second non-example is more relevant to our discussion as it is a complexity 1 pattern over a fixed finite field.

\begin{nonexample}
\label{thm:mono-non-example}
Consider the 4-coloring $\phi \colon \F_5^n \to [4]$ defined as follows: $\phi((0,\ldots,0)) = 1$ and $\phi((0, \dots, 0, x_k, \dots, x_n)) = x_k$ where $x_k \ne 0$ is the first nonzero coordinate. We consider monochromatic solutions to $x+y+z=0$. The only such solution is $x=y=z=0$, but it always remains a monochromatic solution no matter how $\phi$ is modified. Thus there is no removal lemma here either.
\end{nonexample}

This non-example shows that we need to treat zero with care. While not all patterns exhibit the above behavior, a simple way to remove this issue in general seems to be to ignore 0 entirely. One might attempt to consider a less drastic alternative by only ignoring the all-zero solution, but this does not always work, e.g., we can add some additional unconstrained dummy variables to the system. It seems plausible that for every fixed system of patterns one can identify a ``minimal'' notion of a trivial solution that should be ignored, but we take the simpler route here by uniformly ignoring zero. An alternative and weaker possible formulation, which is implied by our main theorem, is to only remove \emph{generic} solutions, i.e., solutions whose span has maximum possible dimension.

As an example of our main result, Theorem~\ref{thm:main}, applied to the above system, we prove:
\begin{quote}
If $\phi\colon \F_5^n\setminus\{0\} \to [4]$ has $o(5^{2n})$ monochromatic solutions to $x+y+z=0$, then one can remove all such solutions by recoloring an $o(1)$-fraction of the space.
\end{quote}

\begin{defn}
An \emph{$r$-colored pattern} $H = (A, \psi)$ is represented by an $\ell\times k$ matrix $A$ with entries in $\FF_q$ along with a coloring $\psi \colon [k] \to [r]$ of the columns of $A$.  Let $V$ be a finite dimensional $\FF_q$-vector space. An \emph{$H$-instance} in a coloring $\phi \colon V \to [r]$ is some $\bm x = (x_1, \dots, x_k) \in V^k$ satisfying $A \bm x = \bm 0$ and $\phi(x_i)= \psi(i)$ for all $i \in [k]$. The \emph{$H$-density} is defined to the number of $H$-instances divided by $\abs{V}^{k - \rank A}$, i.e., the probability that a uniform random $\bm x \in V^k$ satisfying $A\bm x = \bm 0$ also satisfies $\phi(x_i) = \psi(i)$ for all $ i \in [k]$. We say that $\phi$ is \emph{$H$-free} if there are no $H$-instances $\bm x \in (V\setminus\{0\})^k$.
\end{defn}

We generally need to consider a collection of colored patterns.

\begin{example}
	Monochromatic solutions to $x+y+z = 0$ can be encoded by a collection of $r$ separate $r$-colored patterns, $\{(A,\psi_i)\}_{i\in[r]}$, all having the same $A = ( 1 \ 1 \ 1)$ where $\psi_i\equiv i$ is the constant function.
\end{example}

\begin{example}
	A rainbow solution to $x+2y+3z=0$ is one where $x$, $y$, and $z$ have distinct colors, and it can be encoded by a collection of $r(r-1)(r-2)$ separate $H$'s, all having the same $A = (1 \ 2 \ 3)$, and $\psi$ ranging over all injective maps $[3] \to [r]$.
\end{example}

\begin{example}
	The property of having either a blue 3-term arithmetic progression or a red 4-term arithmetic progression can be encoded by two patterns:
	\[
	\paren{A = \begin{pmatrix} 1 & -2  & 1 \end{pmatrix}, \psi \equiv \text{blue}} \quad \text{and}\quad
	\paren{A = \begin{pmatrix}
		1 & -2 & 1 & 0 \\
		0 & 1 & -2 & 1
	\end{pmatrix}, \psi \equiv \text{red}
	}.
	\]
	Note that the second pattern has complexity 2.
\end{example}

\begin{example}
    Suppose that $p$ is prime. We consider maps $\phi\colon\F_p^n\setminus\{0\}\to\F_p$ as $p$-colorings of $\F_p^n$. Consider the set of patterns \[\cH=\{(A,\psi_{a,b,c})\}_{\genfrac{}{}{0pt}{}{a,b,c\in\F_p}{a+b\neq c}}\] where $A=\begin{pmatrix} 1 & 1 & -1\end{pmatrix}$ and $\psi_{a,b,c}$ is defined by  $\psi_{a,b,c}(1)=a, \psi_{a,b,c}(2)=b, \psi_{a,b,c}(3)=c$. Then $\phi\colon\F_p^n\setminus\{0\}\to\F_p$ is linear if and only if it is $\cH$-free.
\end{example}

\subsection{Complexity 1}
\label{ssec:complexity}
Fourier analysis, also known in this context as the Hardy--Littlewood circle method, allows us to analyze certain patterns such as $x+y=z$ or 3-term arithmetic progressions, but not others such as 4-term arithmetic progressions. This distinction is captured by a notion known as the \emph{complexity} of a linear system, introduced by Green and Tao in their work on linear patterns in the primes~\cite{GT10}. Informally, a system has \emph{complexity $s$} it is ``controlled'' by the $U^{s+1}$ Gowers uniformity norm (the \emph{true complexity}~\cite{GW11b} of the system is the smallest such $s$), which we will not define here in general as in this paper we only address complexity 1 systems. 

We consider weighted, or functional, versions of densities of linear patterns. For a finite dimensional $\FF_q$-vector space $V$, an $\ell\times k$ matrix $A$ over $\FF_q$, and functions $f_1, \dots, f_k \colon V \to \C$, we write
\[
\Lambda_A(f_1, \dots, f_k) := \E_{\bm x = (x_1, \dots, x_k) \in V^k : A \bm x = \bm 0} [f_1(x_1)\cdots f_k(x_k)].
\]
Then the $(A, \psi)$-density in $\phi \colon V \to [r]$ is given by $\Lambda_A(1_{\phi^{-1}(\psi(1))},\dots, 1_{\phi^{-1}(\psi(k))})$.

\begin{example} For the matrix $A = (1 \ -2 \ 1)$ corresponding to the pattern $x-2y+z=0$, one has
\[
\Lambda_A(f_1, f_2, f_3) = \E_{x,y \in V}[f_1(x)f_2(y)f_3(2y-x)].
\]
\end{example}

The above expression illustrates the equivalence between expressing a linear pattern as a linear system of equations (e.g., $x-2y+z=0$) and as a collection of linear forms of unconstrained variables (e.g., $(x,y,2y-x)$). The complexity of a system is often defined in the literature in the latter formulation. It is sometimes useful to keep both versions in mind.

The Fourier transform of a function $f \colon V \to \C$ is $\wh f \colon V^* \to \C$ given by
\[
\wh f(\gamma) := \ang{f, \gamma} = \EEE_{x \in V} f(x) \ol{\gamma(x)}
\]
where $\gamma \in V^*$ is a character, i.e., $\gamma \colon V \to \C^\times$ satisfying $\gamma(x+y) = \gamma(x)\gamma(y)$ for all $x,y \in V$. For a finite dimensional $\FF_q$-vector space $V$ of characteristic $p$, one has $V^* \cong V$ by associating $z \in V$ with the character $\gamma_z$ given by $\gamma_z(x) := \exp(  2\pi i \operatorname{tr}(x \cdot z)/p)$ where $x \cdot z := x_1 z_1 + \cdots + x_nz_n$ is the dot product and $\text{tr} \colon \FF_q \to \F_p$ is the trace map. It is fine to think only about the slightly simpler case of prime $q = p$ so that we can write $\wh f (\gamma_z) = \EEE_{x \in V} f(x) e^{-2\pi i (x \cdot z)/p}$ for $z \in V$.

\begin{defn}
We say that an $\ell\times k$ matrix $A$ over $\FF_q$ has \emph{complexity 1} if for every $\epsilon > 0$ there exists $\delta >0$ so that for every finite dimensional $\FF_q$-vector space $V$ and $f_1,\ldots,f_k\colon V \to[-1,1]$, one has
\[
|\Lambda_A(f_1,\ldots,f_k)|\leq\epsilon \quad \text{whenever} \quad  \min_{1\leq i\leq k}\|\wh f_i\|_\infty\leq\delta.
\]
\end{defn}

It is known that a system given by linear forms $L_1, \dots, L_k$ (i.e., generating the null space of $A$) has complexity 1 if and only if the quadratic forms $L_1^2, \dots, L_k^2$ are linearly independent (a technical modification is required if $q$ is even) \cite{GW11a, GW11b, HHL16}.  The general dependence of $\delta$ on $\epsilon$ is a wide open problem, though conjecturally $\delta = \epsilon^{O_A(1)}$ works. See \cite{M18} for some discussion.

\begin{example}
\label{thm:one-equation}
Any single equation $\sum a_ix_i=0$ with at least three non-zero coefficients is a complexity 1 system. In matrix form, this says that any matrix $A$ of dimensions $1\times k$ has complexity 1 as long as it has at least three non-zero entries.
\end{example}

\begin{example}
Let $G=(V,E)$ be a directed graph. Suppose $A$ is a matrix with $|E|$ columns satisfying the property that $\bx=(x_e)_{e\in E}$ satisfies $A\bx=\bm0$ if and only if the signed sum of $x_e$ around each cycle of $G$ is 0. Such an $A$ is complexity 1. This example can be generalized by scaling each column of $A$ by a non-zero constant. In this form Example~\ref{thm:one-equation} is a special case corresponding to the cycle.
\end{example}

\begin{example}
There are complexity 1 systems that do not arise from the above constructions. One example is given by\[A=\begin{pmatrix}
2 & 1 & 1 &-1 & 0 & 0\\
1 & 2 & 1 & 0 &-1 & 0\\
1 & 1 & 2 & 0 & 0 &-1\end{pmatrix}\quad\text{over }\FF_7.\]For those familiar with the terminology, this system has Cauchy-Schwarz complexity 2, but true complexity 1. All systems arising from the two previous examples have true complexity and Cauchy-Schwarz complexity 1.
\end{example}

Now we are ready to state our main result. 

\begin{thm}[Induced arithmetic removal]
\label{thm:main}
Fix a finite set $\cH$ of $r$-colored complexity 1 patterns over $\FF_q$. For every $\epsilon>0$ there exists $\delta=\delta(\epsilon,\cH) > 0$ such that the following holds. Given a finite dimensional $\FF_q$-vector space $V$, if a coloring $\phi \colon V \to [r]$ has $H$-density at most $\delta$ for every $H \in \cH$, then there exists a recoloring $\phi' \colon V\setminus\{0\} \to [r]$ that is $\cH$-free and differs from $\phi$ on at most $\epsilon|V|$ elements of $V$.
\end{thm}

\begin{rem}
The dependence of $1/\delta$ on $1/\epsilon$ is tower-type. Though we write $\delta=\delta(\epsilon,\cH)$ for simplicity, the dependence on $\cH$ is not strong. We could instead write $\delta=\delta(\epsilon,q,r,k)$ where $q$ is the size of the finite field, $r$ is the number of colors used, and $k$ is the maximum number of columns in any of the patterns in $\cH$. 
\end{rem}

For those familiar with the terminology, our main theorem states the following in the language of property testing: every linear-invariant, linear-subspace hereditary property of complexity 1 is testable with constant query complexity and one-sided error. (To be precise one needs the generalization to an infinite number of patterns; see Theorem~\ref{thm:infinite}.)

We conjecture that a similar statement holds for arbitrary patterns (with coefficients in $\ZZ$) in arbitrary abelian groups. However there must be some caveats related to infinite complexity patterns and trivial solutions. As such, we do not state a precise conjecture here. 

\section{Proof strategy}
\label{sec:summary}

We begin by recalling the basic strategy for proving the graph removal lemma. We are given an $n$-vertex graph $G$ with $H$-density $o(1)$ and we wish to make the graph $H$-free by removing $o(n^2)$ edges.

\smallskip
\noindent \textbf{Strategy for graph removal:}
\begin{enumerate}
    \item \textbf{Regularize.} Apply Szemer\'edi's regularity lemma to $G$, partitioning the vertex set into a small number of roughly equal-sized parts, $V(G)=V_1\cup\cdots\cup V_k$, in such a way that $G$ is regular (i.e., quasirandom) between almost all pairs of parts.
    \item \textbf{Clean up.} Modify $G$ by deleting all edges that lie between a pair of parts $(V_i,V_j)$ that either is irregular or has low edge-density.
\end{enumerate}
The resulting graph, after the edge deletions, should be $H$-free since if there were an $H$-subgraph it must sit among high-density regular pairs. By a counting lemma this would imply the existence of many copies of $H$ in the original graph $G$. This would contradict the assumption that $G$ has $H$-density $o(1)$.

For the induced graph removal lemma, the above strategy needs to be further extended, since removing all the edges between an irregular pair $(V_i,V_j)$ may actually create induced $H$-subgraphs where there were none. The new tool here is to use a strong regularity lemma to find a ``regular model'' for $G$ in which every pair of parts is regular.

\smallskip
\noindent \textbf{Strategy for induced graph removal:}
\begin{enumerate}
    \item \textbf{Find a regular model.} Apply a strong regularity lemma to $G$ to produce a partition of the vertex set into a small number of roughly equal-sized parts, $V(G)=V_1\cup\cdots\cup V_k$, and subsets $U_i\subseteq V_i$ each a positive fraction. This is done in such a way that $U:=U_1\cup\cdots\cup U_k$ forms a ``regular model'' for $G$ in the sense that $(U_i,U_j)$ is a regular pair for \emph{every} $i,j$ and that the edge densities between pairs of $V_i$'s and the corresponding pairs of $U_i$'s are close on average.
    \item \textbf{Clean up.} Modify $G$ by deleting all edges that lie between a pair $(V_i,V_j)$ if the edge-density between $(U_i,U_j)$ is close to 0 and by adding all edges that lie between a pair $(V_i,V_j)$ if the edge-density between $(U_i,U_j)$ is close to 1.
\end{enumerate}
The resulting graph, after the edge modifications, should be induced $H$-free for the following reason. Suppose that there were an induced $H$-subgraph sitting among some of the $V_i$'s. Then between the corresponding $U_i$'s, all the pairs are regular and all must have either high edge-density or high non-edge-density as required. By an appropriate counting lemma, this would produce many more induced copies of $H$ in $G$, contradicting the assumption that $G$ has induced $H$-density $o(1)$.

The arithmetic removal lemma is proved analogously to the graph removal lemma. Here a subset $S\subseteq V$ of a finite field vector space $V$ plays the role of a graph and a regular partition of $S\subseteq V$ is given by a subspace $V_1\leq V$ which has bounded codimension. This induces the partition of $V$ into cosets of $V_1$. For this partition to be regular we require $x+V_1$ to be regular (i.e., $S\cap(x+V_1)$ sits inside $x+V_1$ quasirandomly) for almost all $x\in V$.

Suppose $H$ is a complexity 1 system of linear forms over $\FF_q$ that has $o(1)$ density in $S\subseteq V$.

\smallskip
\noindent \textbf{Strategy for arithmetic removal:}
\begin{enumerate}
    \item \textbf{Regularize.} Apply Green's arithmetic regularity lemma to $S\subseteq V$, producing a regular partition determined by a subspace $V_1\leq V$ of bounded codimension.
    \item \textbf{Clean up.} Modify $S$ by deleting all elements that lie in a coset $x+V_1$ that either is irregular or has low density of $S$.
\end{enumerate}
The resulting set, after the deletions, should be $H$-free since if there were an $H$-instance it must sit in high-density regular cosets. By a counting lemma, this would imply the existence of many copies of $H$ in the original set $S$. This would contradict the assumption that $S$ has $H$-density $o(1)$.

A natural attempt to prove the induced arithmetic removal lemma would be to invoke a strong arithmetic regularity lemma to find a regular model. In the arithmetic setting a regular model is a subspace $W$ such that $W\cap (x+V_1)$ is regular for each $x$.

Suppose $H$ is a 2-colored pattern of complexity 1 that has $o(1)$ density in $S\subseteq V$. (Here we view $S$ and $V\setminus S$ as the two color classes.)

\smallskip
\noindent \textbf{(Incorrect) strategy for induced arithmetic removal:}
\begin{enumerate}
    \item \textbf{Find a regular model.} Apply a strong arithmetic regularity lemma to $S\subseteq V$ to produce a partition of $V$ into cosets of $V_1\leq V$ as well as a bounded codimension subspace $W\leq V$ that forms a ``regular model'' for $V$ in the following sense. We require that $W\cap(x+V_1)$ is a regular coset for \emph{every} $x$ and the densities of $S$ in the $x+V_1$'s and the corresponding $W\cap (x+V_1)$'s are close on average. 
    \item \textbf{Clean up.} Modify $S$ by deleting all elements that lie in a coset $x+V_1$ if the density of $S$ in $W\cap(x+V_1)$ is close to 0 and by adding all elements that lie in $x+V_1$ if the density of $S$ in $W\cap(x+V_1)$ is close 1.
\end{enumerate}
If it were possible to carry out this strategy, the resulting set would be induced $H$-free. However, an important difficulty arises that was not present in the proof of the induced graph removal lemma. It turns out that it is impossible to perform the ``find a regular model'' step as the subspace $W\cap V_1$ cannot always be made regular. Green and Sanders \cite{GS16} give an example with the following properties.

\begin{quote}
For all $n$, there exists a subset $S\subset\FF_3^n$ such that for every positive-dimension subspace $V\leq\FF_3^n$ the set $S\cap V$ is \emph{not} an $\epsilon$-regular subset of $V$ for any $\epsilon<\sqrt3/6$.
\end{quote}

Our strategy for circumventing this difficulty is two-fold. First we get as much out of arithmetic regularity as possible: we say that $W$ is a regular model if $W\cap(x+V_1)$ is regular for every $x\in V\setminus V_1$. Next we develop a new tool to deal with $V_1$. Our new tool is a Ramsey-type dichotomy which roughly says the following:
\begin{quote}
    For every finite set  $\cH$ of $2$-colored patterns over $\FF_q$, there exist constants $n_0$ and $\epsilon_0>0$ such that either:
    \begin{enumerate}[(a)]
        \item for every finite dimensional $\FF_q$-vector space $V$ satisfying $\dim V\geq n_0$, every $S\subseteq V$ has $H$-density at least $\epsilon_0$ for some $H\in\cH$; or
        \item for every finite dimensional $\FF_q$-vector space $V$ there exists $S\subseteq V$ that is $\cH$-free.
    \end{enumerate}
\end{quote}
In the above dichotomy we view $S$ and $V\setminus S$ as the two color classes.

Suppose $H$ is a 2-colored pattern of complexity 1 that has $o(1)$ density in $S\subseteq V$.

\smallskip
\noindent \textbf{(Revised) strategy for induced arithmetic removal:}
\begin{enumerate}
    \item \textbf{Find a regular model.} Apply a strong arithmetic regularity lemma to $S\subseteq V$ to produce a partition of $V$ into cosets of $V_1\leq V$ and a bounded codimension subspace $W\leq V$ that forms a ``regular model'' for $V$ in the following sense. We require that $W\cap(x+V_1)$ is a regular coset for \emph{every} $x\not\in V_1$ and the densities of $S$ in the $x+V_1$'s and the corresponding $W\cap (x+V_1)$'s are close on average. 
    \item \textbf{Clean up.} Modify $S$ by deleting all elements that lie in a coset $x+V_1$ if the density of $S$ in $W\cap(x+V_1)$ is close to 0 and by adding all elements that lie in $x+V_1$ if the density of $S$ in $W\cap(x+V_1)$ is close to 1.
    \item \textbf{Patch around the origin.} Since $S\subseteq V$ is a subset with $H$-density $o(1)$, we cannot be in case $(a)$ of the dichotomy above. Therefore there exists a subset $T\subseteq V_1$ that is $H$-free. Replace $S\cap V_1$ with $T$. (One actually needs a slightly more involved application of our Ramsey-type dichotomy.)
\end{enumerate}
Suppose that there is an $H$-instance sitting among some of the $x+V_1$'s. If none of these cosets contains the origin, then all the $W\cap(x+V_1)$ are regular and all either have high $S$-density or high $(V\setminus S)$-density as required. This is impossible since it would imply the existence of many $H$-instances by an appropriate counting lemma. It is also impossible for an $H$-instance to sit entirely in $V_1$ since $V_1$ was patched to be $H$-free. With a little more work we can also rule out the existence of $H$-instances which are contained partially in $V_1$ and partially in $V\setminus V_1$.

\section{Proof of main theorem}
\label{sec:main-thm}

In this section we prove the main theorem, modulo two key tools which are proved in Sections~\ref{sec:rado} and \ref{sec:regularity}. As described in Section~\ref{sec:summary} the proof has three steps: find a regular model; clean up; patch around the origin. We first give the precise definition of regularity. As is typical in this field we work in the weighted, or functional, setting where functions represent weighted colorings.

\begin{defn}
A function $f\colon V\to[-1,1]$ is \emph{$\epsilon$-regular} if $\|\wh{f-\E f}\|_\infty\leq\epsilon$. A coloring $\phi\colon V\to[r]$ is \emph{$\epsilon$-regular} if $1_{\phi^{-1}(i)}$ is $\epsilon$-regular for all $1\leq i\leq r$.
\end{defn}

The next proposition accomplishes the ``find a regular model'' step. As before, $V$ is partitioned into cosets as $V=\bigcup_{x\in U}(x+V_1)$. The regular model will consist of $U\oplus V_2$ where $U$ is a complement for $V_1$ in the sense that $U\oplus V_1=V$. Property (1) states that the model consists of a positive fraction of $V$, property (2) that the colorings of $x+V_1$ and $x+V_2$ have similar densities, and property (3) that $U\oplus V_2$ is regular. This proposition is proved in Section~\ref{sec:regularity} via an application of a strong arithmetic regularity lemma.

\begin{prop}[Regular model]
\label{thm:regularity}
For a finite dimensional $\FF_q$-vector space $V$, functions $f_1, \ldots, f_k \colon\allowbreak V \to [-1,1 ]$, a subspace $V_0\leq V$, and a parameter $\epsilon>0$, there is some $n_{\mathsf{reg'}}=n_{\mathsf{reg'}}(\epsilon,q,k, \codim V_0)$ such that there exist subspaces $V_2\leq V_1\leq V_0$ and a choice of complement $U$ satisfying $U\oplus V_1=V$ such that
\begin{enumerate}
\item $\codim V_2\leq n_{\mathsf{reg'}}$;
\item for all but an $\epsilon$-fraction of $x\in U$ we have\[|\E_{y\in V_1}[f_i(x+y)]-\E_{y\in V_2}[f_i(x+y)]|\leq\epsilon\]for all $1\leq i\leq k$;
\item for each $x\in U\setminus\{0\}$ and $1\leq i\leq k$, the function $f_i|_{x+V_2}$ is $\epsilon$-regular (meaning that the function $g_i\colon V_2\to[-1,1]$ defined by $g_i(y):=f_i(x+y)$ is $\epsilon$-regular).
\end{enumerate}
\end{prop}

With Proposition~\ref{thm:regularity} in hand it is easy to accomplish the ``clean up'' step.

\begin{prop}[Regularity recoloring]
\label{thm:regularity-recoloring}
For a finite dimensional $\FF_q$-vector space $V$, a coloring $\phi\colon V\to[r]$, and parameters $0<\epsilon,\epsilon'\leq1$, there is some $n_{\mathsf{reg}}=n_{\mathsf{reg}}(\epsilon,\epsilon',q,r)$ such that there exist subspaces $V_2\leq V_1\leq V$, a complement $U$ satisfying $U\oplus V_1=V$, and a recoloring $\phi'\colon V\to[r]$ that agrees with $\phi$ on all but an at most $\epsilon|V|$ values satisfying:
\begin{enumerate}
\item $1/\epsilon \leq \codim V_1\leq \codim V_2\leq n_{\mathsf{reg}}$;
\item if a color appears in some coset $x+V_1$ under $\phi'$, then at least an $\epsilon/(2r)$-fraction of $x+V_2$ is that color under $\phi$;
\item for each $x\in U\setminus\{0\}$, the original coloring $\phi|_{x+V_2}$ is $\epsilon'$-regular.
\end{enumerate}
\end{prop}

\begin{proof}
For $1\leq i\leq r$, define $f_i\colon V\to[-1,1]$ by $f_i=1_{\phi^{-1}(i)}$. Let $V_0$ be an arbitrary subspace of codimension $\lceil1/\epsilon\rceil$. Applying Proposition~\ref{thm:regularity} with parameter $\epsilon''=\min(\epsilon/(4r),\epsilon')$ gives subspaces $V_2\leq V_1\leq V_0$ and a choice of complement $U$ satisfying $U\oplus V_1=V$ with the following properties:
\begin{enumerate}[(1$^\prime$)]
\item $\codim V_2\leq n_{\mathsf{reg'}}(\epsilon'',q,r,\lceil 1/\epsilon\rceil)$;
\item for all but an $\epsilon/(4r)$-fraction of $x\in U$ we have\[|\E_{y\in V_1}[f_i(x+y)]-\E_{y\in V_2}[f_i(x+y)]|\leq\frac{\epsilon}{4r}\]for all $1\leq i\leq r$;
\item for each $x\in U\setminus\{0\}$ and $1\leq i\leq r$, the function $f_i|_{x+V_2}$ is $\epsilon'$-regular.
\end{enumerate}

Since $V_1\leq V_0$, clearly $\codim V_1\geq\codim V_0\geq 1/\epsilon$. This proves part (1). Part (3) follows from (3$^\prime$). Next we produce a recoloring $\phi'\colon V\to[r]$ that satisfies part (2). For each $x\in U$, there is some color $i_x$ that appears with density at least $\epsilon/(2r)$ in $x+V_2$. For each color $i$ which appears with density less than $\epsilon/(2r)$ in $x+V_2$, recolor every element of color $i$ in $x+V_1$ to color $i_x$. This recoloring satisfies part (2).

All that remains is to show that $\phi$ and $\phi'$ agree on all but an $\epsilon$-fraction of $V$. At least a $(1-\epsilon/(4r))$-fraction of $x\in U$ satisfy
\begin{equation}
\label{eq:close-density}
|\E_{y\in V_1}[f_i(x+y)]-\E_{y\in V_2}[f_i(x+y)]|\leq\frac{\epsilon}{4r}\quad\text{for all }1\leq i\leq r.\tag{$\ast$}
\end{equation}
For the $x\in U$ failing to satisfy (\ref{eq:close-density}), we use the trivial bound: the number of elements of $x+V_1$ that we recolor is at most all $|V_1|$ of them. For the rest of the $x\in U$, for each color $i$ which occurs with density at most $\epsilon/(2r)$ in $x+V_2$, color $i$ has density at most $\epsilon/(2r)+\epsilon/(4r)$ in $x+V_1$. Thus for each $x\in U$ satisfying (\ref{eq:close-density}) we recolor at most a $3\epsilon/4$-fraction of $x+V_1$. Combining these two cases, $\phi'$ and $\phi$ agree on all but at most a $3\epsilon/4+\epsilon/(4r)\leq\epsilon$-fraction of $V$.
\end{proof}

To accomplish the ``patch around the origin'' step we need the following dichotomy which is proved in Section~\ref{sec:rado}.

\begin{prop}[Ramsey dichotomy]
\label{thm:rado-dichotomy}
Fix a finite set $\cH$ of $r$-colored patterns over $\FF_q$. There exist $n_{\mathsf{rado}}=n_{\mathsf{rado}}(\cH)$ and $\epsilon_{\mathsf{rado}}=\epsilon_{\mathsf{rado}}(\cH)>0$ such that either:
\begin{enumerate}[(a)]
    \item every coloring $\phi\colon V\to[r]$ of $V$, a finite dimensional $\FF_q$-vector space with $\dim V\geq n_{\mathsf{rado}}$, has $H$-density at least $\epsilon_{\mathsf{rado}}$ for some $H\in\cH$; or
    \item for every finite dimensional $\FF_q$-vector space $V$ there exists a coloring $\phi\colon V\setminus\{0\}\to[r]$ that is $\cH$-free.
\end{enumerate}
\end{prop}

As we saw in the previous discussion, Proposition~\ref{thm:regularity} and Proposition~\ref{thm:regularity-recoloring} allow us to recolor $V\setminus V_1$ to be $\cH$-free. Furthermore, applying Proposition~\ref{thm:rado-dichotomy} to $\cH$ allows us to recolor $V_1\setminus\{0\}$ to be $\cH$-free. However this does not rule out $H$-instances which are contained partially in $V_1\setminus\{0\}$ and partially in $V\setminus V_1$.

To deal with this problem we apply Proposition~\ref{thm:rado-dichotomy} not only to $\cH$ but also to some of the \emph{subpatterns} of the patterns $H\in\cH$. The subpatterns of an arithmetic pattern $H$ are analogous to the subgraphs of a graph $H$.

\begin{defn}
Let $H=(A,\psi)$ be an $r$-colored pattern where $A$ is an $\ell\times k$ matrix with entries in $\FF_q$. For $I\subseteq[k]$ say that $H'=(A',\psi|_I)$ is a \emph{subpattern of $H$ restricted to the variables $I$} if for all $V$, a finite dimensional $\FF_q$-vector space, it holds that for $\bx\in V ^I$ we have $A'\bm x=\bm 0$ if and only if there exists $\bm y\in V^k$ satisfying $A\bm y=\bm 0$ and $x_i=y_i$ for all $i\in I$. (Here $\psi|_I\colon I\to[r]$ is the restriction of $\psi$ to $I$.)
\end{defn}

\begin{example}
Consider the following pattern which represents a red 4-term arithmetic progression:
\[
	H=\paren{A = \begin{pmatrix}
		1 & -2 & 1 & 0 \\
		0 & 1 & -2 & 1
	\end{pmatrix}, \psi \equiv \text{red}
	}.
\]	
The subpattern of $H$ corresponding to the first three variables (i.e., $I=\{1,2,3\}$) is a red 3-term arithmetic progression. One possible representation is as:
\[
    H'=\paren{A' = \begin{pmatrix}
		1 & -2 & 1
	\end{pmatrix}, \psi \equiv \text{red}
	}.
\]
\end{example}

\begin{example}
Consider the matrix \[A=\begin{pmatrix}1&1&1&0&0\\0&0&1&1&1\end{pmatrix}.\]The equation $A\bx=\bm0$ imposes the constraints $x_1+x_2+x_3=0$ and $x_3+x_4+x_5=0$. Restricting to variables $x_1,x_2,x_3$ (i.e., $I=\{1,2,3\}$), the only constraint is $x_1+x_2+x_3=0$ which may be represented by the matrix \[A'=\begin{pmatrix} 1&1&1\end{pmatrix}.\] Similarly, restricting to variables $x_1,x_2,x_4,x_5$ (i.e., $I=\{1,2,4,5\}$), the only constraint is $x_1+x_2-x_4-x_5=0$ which may be represented by the matrix\[A'=\begin{pmatrix}1&1&-1&-1\end{pmatrix}.\]
\end{example}

Let $H=(A,\psi)$ be an $r$-colored pattern where $A$ is an $\ell\times k$ matrix with entries in $\FF_q$. Suppose $H'=(A',\psi|_{[j]})$ is a subpattern of $H$ restricted to the first $j$ variables. Our definitions are such that for every finite dimensional $\FF_q$-vector space $V$ and $f_1,\ldots,f_k\colon V\to[-1,1]$, one has\[\Lambda_A(f_1,\ldots,f_j,1,\ldots,1)=\Lambda_{A'}(f_1,\ldots,f_j).\]

\begin{proof}[Proof of Theorem~\ref{thm:main}]
We are given $\mathcal H$, a finite set of $r$-colored complexity 1 patterns over $\FF_q$, and $\phi\colon V\to[r]$, an $r$-coloring of a finite dimensional $\FF_q$-vector space.

Let $\tilde\cH$ be a finite set containing one representative of each subpattern of each $H\in\cH$. More precisely, for each $H=(A,\psi)\in\cH$, a pattern where $A$ is an $\ell\times k$ matrix, and each $I\subseteq[k]$, the set $\tilde \cH$ contains some $H'=(A',\psi|_I)$, a subpattern of $H$ restricted to variables $I$. Define constants\[\epsilon_{\mathsf{rado}}=\min_{\cH'\subseteq\tilde\cH}\epsilon_{\mathsf{rado}}(\cH')\quad\text{and}\quad n_{\mathsf{rado}}=\max_{\cH'\subseteq\tilde\cH}n_{\mathsf{rado}}(\cH').\]
Write $k_{\mathsf{max}}$ for the maximum number of columns in any of the patterns $H\in\cH$. With foresight, we define \[\epsilon_{\mathsf{count}}=\frac{1}{2k_\mathsf{max}}\left(\frac{\epsilon}{4r}\right)^{k_\mathsf{max}}\epsilon_{\mathsf{rado}}.\] Define $\epsilon_{\mathsf{reg}}$ such that for all $H=(A,\psi)\in\tilde\cH$ and for all $f_1,\ldots,f_k\colon V\to[-1,1]$ we have
\begin{equation}
\label{eq:complexity-1}
|\Lambda_{A}(f_1,\ldots,f_k)|\leq\epsilon_{\mathsf{count}}\quad \text{whenever} \quad  \min_{1\leq i\leq k}\|\wh f_i\|_\infty\leq\epsilon_{\mathsf{reg}}.\tag{$\dagger$}
\end{equation}
We can do this since all $H\in\cH$ are complexity 1 (and this implies that all $H\in\tilde\cH$ are complexity 1). Define $n_0=n_{\mathsf{rado}}+n_{\mathsf{reg}}(\epsilon/2,\epsilon_{\mathsf{reg}},q,r)$. Finally define
\[\delta=\min\left(\frac{1}{4}q^{-k_\mathsf{max}\cdot n_{\mathsf{reg}}(\epsilon/2,\epsilon_{\mathsf{reg}},q,r)}\left(\frac{\epsilon}{4r}\right)^{k_\mathsf{max}}\epsilon_{\mathsf{rado}},q^{-n_0k_{\mathsf{max}}}\right).\]

First note that the desired result is trivially true for $\dim V<n_0$; if a pattern $H$ with at most $k_{\mathsf{max}}$ columns has density at most $\delta\leq q^{-n_0k_\mathsf{max}}$ in $\phi\colon V\to[r]$, then $\phi$ must have fewer than one $H$-instance, i.e., be $H$-free.

Now assume that $\dim V\geq n_0$. With these constants chosen we move on to the recoloring algorithm. First, we apply Proposition~\ref{thm:regularity-recoloring} to $\phi$ with parameters $\epsilon/2,\epsilon_{\mathsf{reg}}$. This produces $V_2\leq V_1\leq V$ and $U$ satisfying $U\oplus V_1=V$ and a recoloring $\phi'\colon V\setminus V_1\to[r]$ that differs from $\phi$ on at most $(\epsilon/2)|V|$ elements that satisfies:
\begin{enumerate}
\item $2/\epsilon \leq \codim V_1\leq \codim V_2\leq n_{\mathsf{reg}}(\epsilon/2,\epsilon_{\mathsf{reg}},q,r)$;
\item if a color appears in some coset $x+V_1$ under $\phi'$, then at least an $\epsilon/(4r)$-fraction of $x+V_2$ is that color under $\phi$;
\item for each $x\in U\setminus\{0\}$, the original coloring $\phi|_{x+V_2}$ is $\epsilon_{\mathsf{reg}}$-regular.
\end{enumerate}

Second, define $\cH'\subseteq\tilde\cH$ to be the set of all patterns $H\in\tilde\cH$ which appear with density less than $\epsilon_{\mathsf{rado}}$ in $\phi|_{V_2}$. Note that we have $\dim V_2\geq n_{\mathsf{rado}}$. We now apply Proposition~\ref{thm:rado-dichotomy}; by the definition of $\cH'$, we cannot be in case (a), since $\phi|_{V_2}$ is a counterexample. Thus we are in case (b) which implies the existence coloring $\phi'\colon V_1\setminus\{0\}\to[r]$ that is $\cH'$-free.

Combining these two results gives a recoloring $\phi'\colon V\setminus\{0\}\to[r]$. First note that $\phi'$ differs from $\phi$ on at most $(\epsilon/2)|V|+q^{-\codim V_1}|V|\leq\epsilon|V|$ elements of $V$. Now suppose that $\bx$ is an $H$-instance in $\phi'$ for some $(A,\psi)=H\in\cH$. Write $\bx=\bm u+\bm v$ with $\bm u\in U^k$ and $\bm v\in V_1^k$. We will prove that this implies that the $H$-density is large in $\phi$; we do this by counting $H$-instances of the form $\bm u+\bm y$ with $\bm u$ fixed as above and $\bm y\in V_2^k$. 

Relabeling variables as needed, assume that $u_1,\ldots,u_j=0$ and $u_{j+1},\ldots,u_k\neq 0$. Let $H'=(A',\psi|_{[j]})\in\tilde\cH$ be a subpattern of $H$ restricted to the first $j$ variables. The existence of $\bx$ means that for $i=j+1,\ldots,k$, the color $\psi(i)$ appears at least once in $u_i+V_1$ under $\phi'$ and it means that there exists an $H'$-instance in $\phi'|_{V_1}$, namely $(x_1,\ldots,x_j)$. By property (2) and the $\cH'$-freeness of $\phi'|_{V_1}$, this implies that for $i=j+1,\ldots,k$, the density of color $\psi(i)$ in $u_i+V_2$ under $\phi$ is at least $\epsilon/(4r)$ and that the $H'$-density in $\phi|_{V_2}$ is at least $\epsilon_{\mathsf{rado}}$.

Define $f_i\colon V\to[0,1]$ by $f_i(x)=1_{\phi^{-1}(\psi(i))}(x)$ and $g_i\colon V_2\to[0,1]$ by $g_i(x)=1_{\phi^{-1}(\psi(i))}(x+u_i)$. In this notation, the conclusions of the last paragraph are that $\E g_i\geq \epsilon/(4r)$ for $i=j+1,\ldots,k$ and $\Lambda_{A'}(g_1,\ldots,g_j)\geq \epsilon_{\mathsf{rado}}$. Finally we compute
\begin{align*}
\Lambda_A(f_1,\ldots,f_k)
&\geq q^{-(k-\rank A)\cdot\codim V_2}\Lambda_A(g_1,\ldots,g_k)\\
&\geq q^{-(k-\rank A)\cdot\codim V_2}\left(\Lambda_A(g_1,\ldots,g_j,\E g_{j+1},\ldots, \E g_k)-(k-j)\epsilon_{\mathsf{count}}\right)\\
&\geq q^{-(k-\rank A)\cdot\codim V_2}\left(\left(\frac{\epsilon}{4r}\right)^{k-j}\Lambda_{A'}(g_1,\ldots,g_j)-(k-j)\epsilon_{\mathsf{count}}\right)\\
&\geq q^{-k\cdot\codim V_2}\left(\left(\frac{\epsilon}{4r}\right)^k\epsilon_{\mathsf{rado}}-k\epsilon_{\mathsf{count}}\right)\\
&>\delta.
\end{align*}
The first line represents restricting from arbitrary $H$-instances to $H$-instances of the form $\bm u+\bm y$ with $\bm y\in V_2^k$ (note that the $f_i$ are non-negative). The second line follows by iterating (\ref{eq:complexity-1}). The last line follows from our choice of $\epsilon_{\mathsf{count}}$ and $\delta$.

Thus if the above recoloring procedure does not produce an $\cH$-free coloring, then $\Lambda_A(f_1,\ldots,f_k)>\delta$, meaning that the original coloring had $H$-density more than $\delta$ for some $H\in\cH$. This completes the proof.
\end{proof}

\section{Proof of Proposition~\ref{thm:rado-dichotomy}: a density Ramsey dichotomy}
\label{sec:rado}

In this section we prove Proposition~\ref{thm:rado-dichotomy}. The proof is based on the following trivial dichotomy:
\begin{quote}
    For any $\cH$, a finite set of $r$-colored patterns over $\FF_q$, there exists a constant $n_0$ such that either:
    \begin{enumerate}[(a)]
        \item every $r$-coloring of every $V$ with $\dim V\geq n_0$ contains an $H$-instance for some $H\in\cH$; or
        \item for every finite dimensional $\FF_q$-vector space $V$, there exists an $r$-coloring of $V$ that is $\cH$-free.
    \end{enumerate}
\end{quote}

To boost this result to the full strength of Proposition~\ref{thm:rado-dichotomy} we use a sampling argument. The idea of the sampling is simple. Suppose we have some $\cH$ that is in case (a). Then for an $r$-coloring of $V$ we know that there is not just one $H$-instance, but one $H$-instance in every dimension $n_0$ subspace of $V$. Adding up all of these gives many $H$-instances in $V$. Unfortunately we have overcounted each $H$-instance many times. Double counting more carefully gives us the desired result, but only if the original $H$-instances are \emph{generic} in the following sense.

\begin{defn}
For $A$ an $\ell\times k$ matrix with entries in $\FF_q$ and $\bx=(x_1,\ldots,x_k)\in V^k$ satisfying $A\bx=\bm0$, say that $\bx$ is a \emph{generic solution} if $\dim(\spn\{x_1,\ldots,x_k\})=k-\rank A$. For $H=(A,\psi)$ a colored pattern, define a \emph{generic $H$-instance} to be an $H$-instance $\bx$ that is a generic solution to $A\bx=\bm0$.
\end{defn}

To prove Proposition~\ref{thm:rado-dichotomy} we first show that we can always turn an $H$-instance into a generic $H$-instance, then we perform the sampling argument. 

The first step relies on an argument which is inspired by elements of the proof of Rado's theorem (see \cite{D75}). We reduce the question of whether $\cH$ lies in case (a) of Proposition~\ref{thm:rado-dichotomy} to a finite check. Instead of checking if there exists an $\cH$-instance in every coloring we only need to check the so-called \emph{canonical colorings}. (With some manipulations this check can be made to resemble the ``columns condition'' of Rado's theorem.)

\begin{defn}
For $\chi\colon\FF_q\setminus\{0\}\to[r]$ and $n>0$, define the \emph{$\chi$-canonical coloring} $\Phi_{n,\chi}\colon\FF_q^n\setminus\{0\}\to[r]$ by $\Phi_{n,\chi}(x):=\chi(\fnz(x))$ where $\fnz\colon\FF_q^n\setminus\{0\}\to\FF_q\setminus\{0\}$ maps a vector to its first non-zero coordinate.
\end{defn}

\begin{lemma}
\label{thm:rado}
For $\cH$, a finite set of $r$-colored patterns over $\FF_q$, write $k$ for the maximum number of columns in any of the patterns $H\in\cH$. There exists some $n_{\mathsf{rado}}=n_{\mathsf{rado}}(\cH)$ such that for any $n_0\geq n_{\mathsf{rado}}$ the following are equivalent:
\begin{enumerate}[(a)]
\item For every finite dimensional $\FF_q$-vector space $V$ satisfying $\dim V\geq n_0$ and every coloring $\phi\colon V\setminus\{0\}\to[r]$, there is an $H$-instance in $\phi$ for some $H\in\cH$.
\item For every $\chi\colon\FF_q\setminus\{0\}\to[r]$, there is an $H$-instance in $\Phi_{k,\chi}$ for some $H\in\cH$.
\item For every finite dimensional $\FF_q$-vector space $V$ satisfying $\dim V\geq n_{\mathsf{rado}}$ and every coloring $\phi\colon V\setminus\{0\}\to[r]$, there is a generic $H$-instance in $\phi$ for some $H\in\cH$.
\end{enumerate}
\end{lemma}

Note that in the above lemma statement, an $H$-instance always refers to $\bx$ where none of the vectors $x_i$ are 0 since all of the colorings considered leave 0 uncolored. We will define the function $n_{\mathsf{rado}}(\cH)$ in the statement of Proposition~\ref{thm:rado-dichotomy} to be the same as the $n_{\mathsf{rado}}(\cH)$ in this lemma.

We use the following vector space Ramsey theorem of Graham and Rothschild which was proved with primitive-recursive bounds by Shelah.

\begin{thm}[Graham-Rothschild \cite{GR71}, Shelah \cite{Sh88}\footnote{Note that both of these references prove a slight variant of this result, the ``affine Ramsey theorem''. Spencer \cite{Sp79} gives an easy deduction of this result, referred to as the ``vector space Ramsey theorem'', from the affine Ramsey theorem.}]
There exists $n_{\mathsf{GR}}=n_{\mathsf{GR}}(q,r,k)$ such that for a finite dimensional $\FF_q$-vector space $V$ with $\dim V\geq n_{\mathsf{GR}}$ and $\phi\colon V\setminus\{0\}\to[r]$, an $r$-coloring that is constant on one-dimensional subspaces, there exists a $k$-dimensional subspace $U\leq V$ such that $U\setminus\{0\}$ is monochromatic under $\phi$.
\label{thm:gr}
\end{thm}

The best-known upper bound on $n_{\mathsf{GR}}$ as a function of $k$ is broadly comparable to the third function in the Ackermann hierarchy.

\begin{proof}[Proof of Lemma~\ref{thm:rado}]
Define $n_{\mathsf{rado}}(\cH)=n_{\mathsf{GR}}(q,r^{q-1},2k)$ and fix any $n_0\geq n_{\mathsf{rado}}$. With these choices we show that $(a)\Rightarrow(b)\Rightarrow(c)\Rightarrow(a)$.

$(a)\Rightarrow(b)$: Fix $\chi\colon\FF_q\setminus\{0\}\to[r]$ and take $n\geq n_0$. Assuming (a), there is an $H$-instance in $\Phi_{n,\chi}$ for some $H\in\cH$. Say that $H=(A,\psi)$ where $A$ is an $\ell\times k'$ matrix and call the $H$-instance $\bx\in(\FF_q^n\setminus\{0\})^{k'}$. We use this $\bx$ to find an $H$-instance in $\Phi_{k,\chi}$.

For each $1\leq j\leq k'$, let $b_j$ be the position of the first non-zero coordinate of $x_j$. Relabeling the variables as necessary, assume that $b_1\leq b_2\leq\cdots\leq b_{k'}$. Consider the map $\Pi\colon\FF_q^n\to\FF_q^k$ which sends $v=(v_1,\ldots,v_n)\in\FF_q^n$ to $(v_{b_1},v_{b_2},\ldots,v_{b_{k'}},0,\ldots,0)\in\FF_q^k$, where there are $k-k'$ zeros added at the end. One can easily check that each $x_j$ satisfies $\Phi_{n,\chi}(x_j)=\Phi_{k,\chi}(\Pi(x_j))$. Thus $\Pi(\bx)$ is an $H$-instance in $\Phi_{k,\chi}$, as desired.

$(b)\Rightarrow(c)$: Take $n\geq n_{\mathsf{rado}}$ and a coloring $\phi\colon V\setminus\{0\}\to[r]$ of a finite dimensional $\FF_q$-vector space $V$ satisfying $\dim V=n$. We define an $r^{q-1}$-coloring $\overline{\phi}\colon V\setminus\{0\}\to[r]^{q-1}$ that is constant on one-dimensional subspaces as follows. Pick an isomorphism $V\cong \FF_q^n$. For $x\in V\setminus\{0\}$ there exists a unique $a\in\FF_q\setminus\{0\}$ such that the first non-zero coordinate of $ax$ is 1. Then define \[\overline{\phi}(x):=(\phi(abx))_{b\in\FF_q\setminus\{0\}}\in[r]^{q-1}.\] By Theorem~\ref{thm:gr} and our choice of $n_{\mathsf{rado}}$, there exists $U\leq V$, a $2k$-dimensional subspace such that $\overline\phi$ is constant on $U\setminus\{0\}$. We claim that this implies that there is a coloring $\chi\colon\FF_q\setminus\{0\}\to[r]$ and an isomorphism $\iota\colon U\overset{\sim}{\to}\FF_q^{2k}$ such that $\phi$ agrees with the canonical coloring $\Phi_{2k,\chi}$ on $U\setminus\{0\}$.

The color of $U\setminus\{0\}$ under $\overline\phi$ is some vector in $[r]^{q-1}$, in other words, a function $\FF_q\setminus\{0\}\to[r]$. This is our $\chi$. Now pick a basis for $U\leq\FF_q^n$. Performing Gaussian elimination on this basis transforms it into a basis $u_1,\ldots,u_{2k}\in U$ such that the first non-zero coordinate of each of these vectors is 1 and the location of the first non-zero coordinates is strictly increasing from $u_1$ to $u_{2k}$. Then define $\iota\colon U\overset{\sim}{\to}\FF_q^{2k}$ to send $c_1u_1+\cdots+c_{2k}u_{2k}\mapsto(c_1,\ldots,c_{2k})\in\FF_q^{2k}$. These $\chi,\iota$ have the desired properties since $\fnz(c_1u_1+\cdots+c_{2k}u_{2k})=\fnz((c_1,\ldots,c_{2k}))$.

Consider $\FF_q^k\oplus\{0\}\subset\FF_q^{2k}$, the subspace of vectors of the form $(\ast,\ldots,\ast,0,\ldots,0)$. Assuming (b), there exists an $H$-instance $\bx\in((\FF_q^k\setminus\{0\})\oplus\{0\})^{k'}$ in $\Phi_{2k,\chi}$. Now for any $\bm y\in\left(\{0\}\oplus\FF_q^k\right)^{k'}$ we have that $x_j$ and $x_j+y_j$ are the same color under $\Phi_{2k,\chi}$ for all $j$. Pick $\bm y$ any generic solution to $A\bm y=\bm0$ in $\{0\}\oplus\FF_q^k$. Then $\bx+\bm y$ is a generic $H$-instance in $\Phi_{2k,\chi}$. Since $\phi|_U$ agrees with $\Phi_{2k,\chi}$, we have found a generic $H$-instance in $\phi$, as desired.

$(c)\Rightarrow(a)$: Obvious.
\end{proof}

\begin{proof}[Proof of Proposition~\ref{thm:rado-dichotomy}]
Fix a finite set $\cH$ of $r$-colored patterns over $\FF_q$. Write $k_{\mathsf{max}}$ for the maximum number of columns in any of the patterns $H\in\cH$. Define $n_{\mathsf{rado}}=n_{\mathsf{rado}}(\cH)$ to be the same as in Lemma~\ref{thm:rado} and, with foresight, define\[\epsilon_{\mathsf{rado}}=\frac1{10|\cH|}q^{-n_{\mathsf{rado}}k_{\mathsf{max}}}.\]

Suppose that we are not in case (b) of Proposition~\ref{thm:rado-dichotomy}; namely, that there is some finite dimensional $\FF_q$-vector space $V_0$ such that every coloring $\phi\colon V_0\setminus\{0\}\to[r]$ contains an $H$-instance for some $H\in\cH$. We now apply Lemma~\ref{thm:rado} to $\cH$ with parameter $n_0:=\max\{\dim V_0,n_{\mathsf{rado}}\}$. By assumption, statement (a) of Lemma~\ref{thm:rado} is true, so we conclude that statement (c) is also true: for every finite dimensional $\FF_q$-vector space $V$ satisfying $\dim V\geq n_{\mathsf{rado}}$ and every coloring $\phi\colon V\setminus\{0\}\to[r]$, there is a generic $H$-instance in $\phi$ for some $H\in\cH$.

Now we prove we are in case (a) of Proposition~\ref{thm:rado-dichotomy}. Fix $V$, a finite dimensional $\FF_q$-vector space with $\dim V = n \geq n_{\mathsf{rado}}$, and a coloring $\phi\colon V \to [r]$. We count the total number of generic $\cH$-instances in $\phi$. For ease of notation, given $H=(A,\psi)\in\cH$, write $k_H$ for the number of columns in $A$.

Consider the following random process: pick a uniform random $H=(A,\psi)\in\cH$ and then pick a uniform random generic solution $A\bx=0$ with $\bx\in V^{k_H}$. Write $N_H(\phi)$ for the number of generic $H$-instances in $\phi$. Then the probability that this process produces a pair $(H,\bx)$ where $\bx$ is a generic $H$-instance is \[\frac1{|\cH|}\sum_{H=(A,\psi)\in\cH}\frac{N_H(\phi)}{(q^n-1)(q^n-q)\cdots(q^n-q^{k_H-\rank A-1})}\leq\frac{10}{|\cH|}\sum_{H=(A,\psi)\in\cH}N_H(\phi)q^{-n(k_H-\rank A)}.\]

Now consider the following random process: pick $U\leq V$, a uniform random subspace of dimension $n_{\mathsf{rado}}$, pick a uniform random $H=(A,\psi)\in\cH$, and then pick a uniform random generic solution $A\bx=0$ with $\bx\in U^{k_H}$. Since, by assumption, there is at least one generic $\cH$-instance in $\phi|_U$, the probability that this process finds a pair $(H,\bx)$ where $\bx$ is a generic $H$-instance is at least\[\frac1{|\cH|}\frac1{(q^{n_{\mathsf{rado}}}-1)(q^{n_{\mathsf{rado}}}-q)\cdots(q^{n_{\mathsf{rado}}}-q^{k_{\mathsf{max}}-1})}\geq\frac{1}{|\cH|}q^{-n_{\mathsf{rado}}k_{\mathsf{max}}}.\]

Note that these two random processes actually produce the same distribution on pairs $(H,\bm x)$. (This is the key place where we use the assumption that the $\bx$ are generic solutions.) Combining these two bounds, we conclude that\[\sum_{H=(A,\psi)\in\cH}N_H(\phi)q^{-n(k_H-\rank A)}\geq\frac{1}{10} q^{-n_{\mathsf{rado}}k_{\mathsf{max}}}.\] Thus by the pigeonhole principle, there exists some $H=(A,\psi)\in\cH$ such that $N_H(\phi)q^{-n(k_H-\rank A)}$, i.e., the $H$-density in $\phi$, is at least $\epsilon_{\mathsf{rado}}$, as desired.
\end{proof}

\section{Proof of Proposition~\ref{thm:regularity}: arithmetic regularity lemmas}
\label{sec:regularity}

We give two proofs of Proposition~\ref{thm:regularity}, giving wowzer- and tower-type bounds respectively. The first proof can be considered the arithmetic analog of the strong regularity argument in \cite{AFKS00} and the second bears some resemblance to an arithmetic analog of the cylinder regularity argument in \cite[Theorem 1.3]{CF12}.

Our main technique for proving arithmetic regularity lemmas is the notion of energy.

\begin{defn}
Fix a finite dimensional $\FF_q$-vector space $V$ and functions $f_1,\ldots,f_k\colon V\to[-1,1]$. Given $\cP$ a partition of $S\subseteq V$, define $(f_i)_{\cP}\colon S\to[-1,1]$ to be the projection of $f_i|_S$ onto the $\sigma$-algebra generated by $\cP$. In other words, $(f_i)_\cP(x)$ is defined to be the average of $f_i$ over the part of $\cP$ containing $x$. The \emph{energy} of the partition is defined to be\[\cE(\cP):=\sum_{i=1}^k\norm{(f_i)_\cP}_{L^2(S)}^2.\]
One case will be of particular interest to us. Let $V_1\leq V$ be a linear subspace and $S\subseteq V$ any set which can be written as the union of cosets of $V_1$. Then write $\cP(V_1|S)$ for the partition of $S$ into cosets of $V_1$. When $S$ is not given, it is assumed to be all of $V$.
\end{defn}

We record the following three properties of energy.

\begin{prop}
\label{thm:energy}
For a finite dimensional $\FF_q$-vector space $V$ and functions $f_1,\ldots,f_k\colon V\to[-1,1]$, we have:
\begin{enumerate}
\item for any partition $\cP$ of a set $S\subseteq V$,\[0\leq\cE(\cP)\leq k;\]
\item for any partitions $\cP,\cQ$ of a set $S\subseteq V$ with $\cQ$ refining $\cP$ (written $\cQ\succeq\cP$),\[\cE(\cQ)-\cE(\cP)=\sum_{i=1}^k\norm{(f_i)_{\cQ}-(f_i)_{\cP}}_{L^2(S)}^2\geq0.\]
\item for $V_1\leq V$, if there exists $1\leq i\leq k$ and $x\in V$ such that $f_i|_{x+V_1}$ is not $\epsilon$-regular, there exists $V_2\leq V_1$ of codimension 1 such that\[\cE(\cP(V_2|x+V_1))-\cE(\cP(V_1|x+V_1))>\epsilon^2.\]
\end{enumerate}
\end{prop}

\begin{proof}
The first result follows since trivially $0\leq\norm{(f_i)_{\cP}}_{L^2(S)}^2\leq 1$.

The second result follows from the Pythagorean theorem: $\norm{(f_i)_{\cP}}_{L^2(S)}^2+\norm{(f_i)_{\cQ}-(f_i)_{\cP}}_{L^2(S)}^2=\norm{(f_i)_{\cQ}}_{L^2(S)}^2$ as long as $(f_i)_{\cP}$ and $(f_i)_{\cQ}-(f_i)_{\cP}$ are orthogonal. We compute \[\ang{(f_i)_{\cP},(f_i)_{\cQ}-(f_i)_{\cP}}_{L^2(S)}=\ang{f_i,\left((f_i)_{\cQ}-(f_i)_{\cP}\right)_{\cP}}_{L^2(S)}=\ang{f_i,0}_{L^2(S)}=0.\]

To prove the third result, suppose that $\lvert\wh{f_i|_{x+V_1}}(\gamma_z)\rvert>\epsilon$ for some $\gamma_z\in V_1^*\setminus\{0\}$. Recall that $\gamma_z$ is the character $\gamma_z(x):=\exp(  2\pi i \operatorname{tr}(x \cdot z)/p)$. Define $V_2\leq V_1$ to be $V_2=\{y\in V_1:y\cdot z=0\}$. For ease of notation, write $\cP$ for $\cP(V_1|x+V_1)$ and $\cQ$ for $\cP(V_2|x+V_1)$ for the rest of the proof. Note that $\gamma_z$ is constant on parts of $\cQ$. Thus we can write
\begin{align*}
\epsilon^2
&<\abs{\wh{f_i|_{x+V_1}}(\gamma_z)}^2\\
&=\abs{\ang{f_i,\gamma_z}_{L^2(x+V_1)}}^2\\
&=\abs{\ang{(f_i)_{\cQ},\gamma_z}_{L^2(x+V_1)}}^2\\
&=\abs{\ang{(f_i)_{\cQ}-(f_i)_{\cP},\gamma_z}_{L^2(x+V_1)}}^2 && \text{(since $\E\gamma_z=0$ and $(f_i)_{\cP}$ is constant)}\\
&\leq\norm{(f_i)_{\cQ}-(f_i)_{\cP}}_{L^2(x+V_1)}^2 && \text{(Cauchy-Schwarz inequality)}\\
&\leq\cE(\cQ)-\cE(\cP). && \qedhere
\end{align*}
\end{proof}

\subsection{Proof 1: Strong arithmetic regularity}

We first prove Green's arithmetic regularity lemma in finite field vector spaces by the standard energy increment method. We then iterate this arithmetic regularity lemma to prove a strong arithmetic regularity lemma. Given functions $f_1,\ldots,f_k\colon V\to[-1,1]$, a subspace $V_0\leq V$, and $\epsilon>0$, this strong arithmetic regularity lemma gives subspaces $V_2\leq V_1\leq V_0$ with several desirable properties. To complete the proof of Proposition~\ref{thm:regularity} we choose a random complement $U$ satisfying $U\oplus V_1= V$ and show that $U$ has the desired properties with positive probability.

For completeness we include the standard proof of Green's arithmetic regularity lemma.

\begin{thm}[Green's arithmetic regularity]
\label{thm:green-reg}
For a finite dimensional $\FF_q$-vector space $V$, functions $f_1,\ldots,f_k\colon V\to[-1,1]$, a subspace $V_0\leq V$, and a parameter $\epsilon>0$, there exists $n_{\mathsf{green-reg}} = n_{\mathsf{green-reg}}(\epsilon,q,k,\codim V_0)$ such that there exists a subspace $V_1\leq V_0$ such that
\begin{enumerate}
\item $\codim V_1\leq n_{\mathsf{green-reg}}$;
\item for each $1\leq i\leq k$, for all but an $\epsilon$-fraction of $x\in V$ the function $f_i|_{x+V_1}$ is $\epsilon$-regular.
\end{enumerate}
\end{thm}

\begin{proof}
We iterate to produce a sequence of subspaces $V_0\geq V_1\geq\cdots\geq V_M$ such that $V_M$ has the desired properties. This sequence will satisfy $\codim V_{m+1}\leq \codim V_m+q^{\codim V_m}$ and $\mathcal E(\cP(V_{m+1}))>\mathcal E(\cP(V_m))+\epsilon^3$.

Suppose that we have produced some $V_m$ that does not satisfy property (2) above. Pick some complement $U$ satisfying $U\oplus V_m=V$. Then there exist some $1\leq i\leq k$ and $\epsilon |U|$ values of $x\in U$, say $U'\subset U$, such that $f_i|_{x+V_m}$ is not $\epsilon$-regular for each $x\in U'$.

Fix $x\in U'$. Since $f_i|_{x+V_m}$ is not $\epsilon$-regular, part (3) of Proposition~\ref{thm:energy} gives a subspace $V(x)\leq V_m$ of codimension 1 such that $\mathcal E(\cP(V(x)|x+V_m))>\mathcal E(\cP(V_m|x+V_m))+\epsilon^2$. Define $V_{m+1}$ to be the intersection of all of the $V(x)$'s. Then $\codim V_{m+1}\leq \codim V_m+|U'|\leq \codim V_m+q^{\codim V_m}$, as desired.

Note that for $W\leq V_m$, we have $\mathcal E(\cP(W))=\E_{x\in U} [\mathcal E(\cP(W|x+V_m))]$.  Therefore
\begin{align*}
\mathcal E(\cP(V_{m+1}))-\mathcal E(\cP(V_m))
&=\E_{x\in U}\left[\mathcal E(\cP(V_{m+1}|x+V_m))-\mathcal E(\cP(V_m|x+V_m))\right]\\
&\geq\frac{|U'|}{|U|}\E_{x\in U'}\left[\mathcal E(\cP(V(x)|x+V_m))-\mathcal E(\cP(V_m|x+V_m))\right]\\
&>\epsilon^3.
\end{align*}

Since $0\leq \mathcal E(\cP)\leq k$ for all $\cP$, we conclude that $M\leq k\epsilon^{-3}$. Thus $V_M$ has the desired properties where $n_{\mathsf{green-reg}}$ grows at a rate comparable to a tower of $q$'s of height $k\epsilon^{-3}$.
\end{proof}

\begin{thm}[Strong arithmetic regularity]
\label{thm:strong-reg}
For a finite dimensional $\FF_q$-vector space $V$, functions $f_1,\ldots,f_k\colon V\to[-1,1]$, a subspace $V_0\leq V$, and parameters $\delta>0$ and $\bm\epsilon=(\epsilon_0,\epsilon_1,\ldots)$ satisfying $\epsilon_0\geq\epsilon_1\geq\cdots>0$, there exists $n_{\mathsf{strong-reg}}=n_{\mathsf{strong-reg}}(\delta,\bm\epsilon,q,k,\codim V_0)$ such that there exist subspaces $V_2\leq V_1\leq V_0$ such that
\begin{enumerate}
\item $\codim V_2\leq n_{\mathsf{strong-reg}}$;
\item $\mathcal E(\cP(V_2))\leq\mathcal E(\cP(V_1))+\delta$;
\item for each $1\leq i\leq k$, for all but an $\epsilon_{\codim V_1}$-fraction of $x\in  V$ the function $f_i|_{x+V_2}$ is $\epsilon_{\codim V_1}$-regular.
\end{enumerate}
\end{thm}

\begin{proof}
We iterate Green's arithmetic regularity lemma to produce a sequence of subspaces $ V_0=V^{(0)}\geq V^{(1)}\geq\cdots\geq V^{(M)}$ such that setting $V_1=V^{(M-1)}$ and $V_2=V^{(M)}$ proves the theorem. Let $V^{(m+1)}$ be the result of Theorem~\ref{thm:green-reg} to $V^{(m)}$ with parameter $\epsilon_{\codim V_m}$.

Note that $0\leq \mathcal E(\cP(V^{(0)}))\leq\mathcal E(\cP(V^{(1)}))\leq\cdots\leq \mathcal E(\cP(V^{(M)}))\leq k$. Therefore we can stop the iteration at $M\leq k\delta^{-1}$ such that $\mathcal E(\cP(V^{(M)}))-\mathcal E(\cP(V^{(M-1)}))\leq \delta$, as desired.
\end{proof}

\begin{proof}[Proof of Proposition~\ref{thm:regularity}]
We can assume that $|V_0|\leq(\epsilon/4)|V|$. (If this does not hold, simply replace $V_0$ with a subspace of codimension at most $\lceil\log_q(4\epsilon^{-1})\rceil$; if this is impossible, then the desired result is trivially true by setting $V_1=V_2=\{0\}$.)

Apply Theorem~\ref{thm:strong-reg} to $f_1,\ldots,f_k\colon V\to[-1,1]$ and $V_0\leq V$ with parameters $\delta=\epsilon^3/4$ and $\epsilon_m=\min(\epsilon,q^{-m}/(2k))$. Write $V_2\leq V_1\leq V_0$ for the subspaces produced. We claim that a uniform random choice of $U\leq V$ satisfying $U\oplus V_1=V$ also satisfies the desired properties with positive probability.

For each $i$, pick $x\in V$ uniformly at random. Property (3) implies that with probability at least $1-\epsilon_{\codim V_1}$, the function $f_i|_{x+V_2}$ is $\epsilon$-regular. By the union bound, with probability at least $1-\epsilon_{\codim V_1}\cdot k\cdot (|U|-1)>1/2$, for every $x\in U\setminus\{0\}$ and every $1\leq i\leq k$, the function $f_i|_{x+V_2}$ is $\epsilon$-regular.

Now by part (2) of Proposition~\ref{thm:energy}, we write
\begin{align*}
\epsilon^3/4
&\geq \mathcal E(\cP(V_2))-\mathcal E(\cP(V_1))\\
&=\sum_{i=1}^k\norm{(f_i)_{\cP(V_2)}-(f_i)_{\cP(V_1)}}_{L^2(V)}^2\\
&=\E_{x\in  V}\left[\sum_{i=1}^k\left(\E_{y\in V_2}[f_i(x+y)]-\E_{y\in V_1}[f_i(x+y)]\right)^2\right].
\end{align*}

This implies that for a uniform random choice of $x\in V$, with probability at least $1-\epsilon/4$, this $x$ satisfies
\begin{equation}
\label{eq:good-x}
|\E_{y\in V_2}[f_i(x+y)]-\E_{y\in V_1}[f_i(x+y)]|\leq \epsilon\quad\text{for all}\quad 1\leq i\leq k.\tag{$\ddagger$}
\end{equation}
Therefore choosing $U$ randomly, the expected number of $x\in U$ failing to satisfy (\ref{eq:good-x}) is at most $1+(\epsilon/4)(|U|-1)<(\epsilon/2)|U|$. By Markov's inequality, with probability at least $1/2$, at most an $\epsilon$-fraction of $x\in U$ fail to satisfy (\ref{eq:good-x}). Thus by the union bound, $U$ has the desired properties with positive probability.
\end{proof}

\subsection{Proof 2: An improved bound}

We prove a weak arithmetic regularity lemma with exponential bounds and then iterate this lemma to produce a strong arithmetic regularity lemma with tower-type bounds that is just strong enough to prove Proposition~\ref{thm:regularity}.

\begin{defn}
For a finite dimensional $\FF_q$-vector space $V$ and $U\leq V$, a \emph{decomposition of $V$ with respect to $U$} is a set of subspaces $\mathcal D=\{W_1,\ldots,W_\ell\}$ all of the same codimension and all transverse to $U$ (i.e., $W_i\cup U$ spans $V$ for each $i$) such that $W_1\setminus U,\ldots,W_\ell\setminus U$ are disjoint sets that together partition $ V\setminus U$. We write $\codim \mathcal D:=\codim W_1=\cdots=\codim W_\ell$.
\end{defn}

One can easily see that every decomposition $\cD$ has size\[|\cD|=\frac{q^{\dim V}-q^{\dim U}}{q^{\dim V-\codim\cD}-q^{\dim U-\codim\cD}}=q^{\codim\cD}.\]

To build intuition, we give several constructions of decompositions of increasing complexity. We start with a very simple example.

\begin{example}
\label{thm:decomp-1}
Take $V=\FF_q^2$ and $U=\{(0,x):x\in\FF_q\}$. We define $\cD$, a decomposition of $V$ with respect to $U$, to be the set of all 1-dimensional subspaces of $V$ other than $U$ itself. More concretely, $\cD=\{W_a\}_{a\in\FF_q}$ where $W_a=\{(x,ax):x\in\FF_q\}$.
\end{example}

Next we give a slightly more complicated example.

\begin{example}
\label{thm:decomp-2}
Take $V=\FF_q^{2n}$ and $U=\{0\}\oplus\FF_q^n$. We view $V\cong\F_{q^n}\oplus\F_{q^n}$, an isomorphism of $\FF_q$-vector spaces. Then we take the same construction as the previous example: $\cD=\{W_a\}_{a\in\F_{q^n}}$ where $W_a=\{(x,ax):x\in\F_{q^n}\}$.
\end{example}

Both these examples are very special: they both have the property that $W\cap U=\{0\}$ for each $W\in\cD$. We can make our examples slightly more interesting as follows.

\begin{example}
\label{thm:decomp-3}
Take any $U'\leq U\leq V$ with $\dim V-\dim U=\dim U-\dim U'$. The previous example gives a decomposition $\cD$ of $V/U'$ with respect to $U/U'$. Then $\cD'=\{W\oplus U'\}_{W\in\cD}$ is a decomposition of $V$ with respect to $U$.
\end{example}

The decompositions $\cD$ produced by Example~\ref{thm:decomp-3} are slightly more interesting, but they satisfy $W\cap U=U'$ for each $W\in\cD$. To prove our decomposition regularity lemma, we will need to consider more complicated decompositions.

For two decompositions $\cD,\cD'$ of $V$ with respect to $U$, say that $\cD'$ refines $\cD$ (written as $\cD'\succeq \cD$) if the partition $\bigcup_{W\in\mathcal D'}(W\setminus U)$ refines the partition $\bigcup_{W\in\mathcal D}(W\setminus U)$.

We can create more complicated examples of decompositions of $V$ with respect to $U$ as follows. We construct a sequence of decompositions $\{V\}=\mathcal D_0\preceq\mathcal D_1\preceq\cdots$ iteratively. To get from $\mathcal D_m$ to $\mathcal D_{m+1}$, we refine each part of $\mathcal D_m$. Fix $W\in\cD_m$ and pick any subspace $U'\leq U\cap W$ such that $\dim W-\dim (U\cap W)=\dim (U\cap W)-\dim U'$. Example~\ref{thm:decomp-3} gives a decomposition of $W$ with respect to $U$. Replacing each part $W\in\cD_m$ with the subspaces produces this way gives a new decomposition $\cD_{m+1}$.

Fix $f_1,\ldots,f_k\colon V\to[-1,1]$ and a subspace $U\leq V$. For a decomposition $\cD$ of $V$ with respect to $U$, we associate to it the partition $\cP(\cD)$ whose parts are, for each $W\in \cD$, the cosets of $W\cap U$ contained in $W\setminus U$. As in the previous section, we will consider the energy \[\cE(\cP(\cD))=\sum_{i=1}^k\norm{(f_i)_{\cP}}_{L^2(V\setminus U)}^2.\]

\begin{figure}[t]
\centering
\begin{subfigure}{0.4\textwidth}
\centering
\begin{tikzpicture}
\draw[very thick] (0,0) -- (4,0) -- (4,4) -- (0,4) -- cycle;
\draw[very thick] (0,1) -- (4,1);
\draw[very thick] (0,2) -- (4,2);
\draw[very thick] (0,3) -- (4,3);
\end{tikzpicture}
\caption{An initial partition.\vspace{12 pt}}
\label{fig:a}
\end{subfigure}
\qquad
\begin{subfigure}{0.4\textwidth}
\centering
\begin{tikzpicture}
\draw[very thick] (0,0) -- (4,0) -- (4,4) -- (0,4) -- cycle;
\draw[very thick] (0,1) -- (4,1);
\draw[very thick] (0,2) -- (4,2);
\draw[very thick] (0,3) -- (4,3);
\draw (0,3.5) -- (4,3.5);
\draw (2,2) -- (2,3);
\draw (0,1) -- (4,2);
\end{tikzpicture}
\caption{Refining some parts of the initial partition increases the energy.}
\label{fig:b}
\end{subfigure}
\begin{subfigure}{0.4\textwidth}
\centering
\begin{tikzpicture}
\draw[very thick] (0,0) -- (4,0) -- (4,4) -- (0,4) -- cycle;
\draw[very thick] (0,1) -- (4,1);
\draw[very thick] (0,2) -- (4,2);
\draw[very thick] (0,3) -- (4,3);
\draw (0,0.5) -- (4,0.5);
\draw (0,1.5) -- (4,1.5);
\draw (0,2.5) -- (4,2.5);
\draw (0,3.5) -- (4,3.5);
\draw (2,0) -- (2,4);
\draw (2,0) -- (4,0.5);
\draw (0,0) -- (4,1);
\draw (0,0.5) -- (4,1.5);
\draw (0,1) -- (4,2);
\draw (0,1.5) -- (4,2.5);
\draw (0,2) -- (4,3);
\draw (0,2.5) -- (4,3.5);
\draw (0,3) -- (4,4);
\draw (0,3.5) -- (2,4);
\end{tikzpicture}
\caption{Taking the common refinement produces a partition with many parts.}
\label{fig:c}
\end{subfigure}
\qquad
\begin{subfigure}{0.4\textwidth}
\centering
\begin{tikzpicture}
\draw[very thick] (0,0) -- (4,0) -- (4,4) -- (0,4) -- cycle;
\draw[very thick] (0,1) -- (4,1);
\draw[very thick] (0,2) -- (4,2);
\draw[very thick] (0,3) -- (4,3);
\draw (0,3.5) -- (4,3.5);
\draw (2,2) -- (2,3);
\draw (0,1) -- (4,2);
\draw (2,0) -- (2,1);
\end{tikzpicture}
\caption{A minimal decomposition which contains \subref{fig:b} and has many fewer parts. (The bottom part is partitioned arbitrarily.)}
\label{fig:d}
\end{subfigure}
\caption{An example of refining partitions and decompositions.}
\label{fig:ex}
\end{figure}
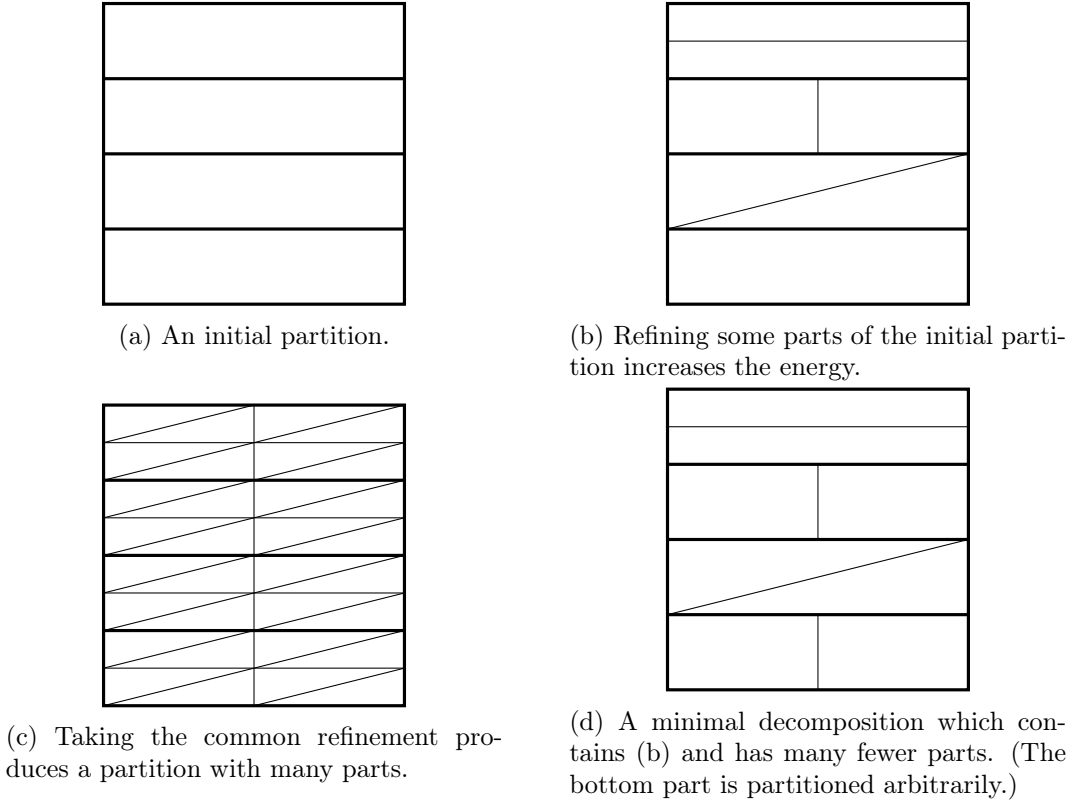

Now we look at an example, depicted in Figure~\ref{fig:ex}, to explain why using decompositions can help us prove a more efficient regularity lemma. Suppose $U\leq\F_2^n$ is a codimension 1 subspace. Then $\F_2^n\setminus U$ is just $(n-1)$-dimensional affine space. Thus a decomposition of $\F_2^n$ with respect to $U$ is simply a partition of $\F_2^{n-1}$ into affine subspaces of the same codimension. Contrast this to the situation in the previous subsection where we we partitioned our space into affine subspaces which were all cosets of the same subspace.

This freedom allows us to prove a more efficient regularity lemma. Suppose after some number of iterations we have produced a partition of $\F_2^{n-1}$ depicted in Figure~\ref{fig:a}. In both the proof of arithmetic regularity and our weak regularity lemma, the next step is to find parts of the partition for which the functions $f_1,\ldots,f_k$ are not regular and find a refinement of each of these parts which increases the energy. Suppose that the first three parts have refinements which increase the energy as depicted in Figure~\ref{fig:b}.

If we want our next partition to be made of cosets of the same subspace (as in the proof of arithmetic regularity) we need to take the common refinement, pictured in Figure~\ref{fig:c}. However, if all we want is to produce a decomposition (as will be the case in the proof of our weak arithmetic regularity lemma) we simply have to arbitrarily partition the remaining parts to make every part the same codimension, pictured in Figure~\ref{fig:d}.

Therefore we can prove our weak arithmetic regularity via an iteration where the codimension of our decomposition only increases by a constant each iteration where in proving arithmetic regularity the codimension of the partition increases by an exponential at every step of the iteration.

For the case of general $U\leq V$ the proof is slightly more complicated, but the main idea is the same. Instead of a picture like the one above, we have $q^{\codim U}-1$ similar copies to work with (the cosets of $U$). Since $U$ is not codimension 1, we cannot use Example~\ref{thm:decomp-1} at each step to refine, but instead have to use Example~\ref{thm:decomp-2}. In addition, each step of the iteration produces a smaller energy increment ($\epsilon^2/q^{\codim U}$ instead of $\epsilon^2$), but this increment still produces exponential-type bounds instead of tower-type bounds.

Now we begin the proof.

\begin{thm}[Weak decomposition regularity]
\label{thm:weak-reg}
For a finite dimensional $\FF_q$-vector space $V$, functions $f_1,\ldots,f_k\colon V\to[-1,1]$, a subspace $U\leq V$, and a parameter $\epsilon>0$, if $\dim V\geq \codim U\cdot q^{\codim U} k\epsilon^{-3}$, then there exists a decomposition $\cD$ such that
\begin{enumerate}
\item $\codim \cD\leq \codim U\cdot q^{\codim U} k \epsilon^{-3}$;
\item for each $1\leq i\leq k$, for all but an $\epsilon$-fraction of $W\in\cD$, for all $x\in W\setminus U$ the function $f_i|_{x+W\cap U}$ is $\epsilon$-regular.
\end{enumerate}
\end{thm}

\begin{proof}
We iterate to produce a sequence of decompositions $\{V\}=\cD_0\preceq \cD_1\preceq\cdots\preceq \cD_M$ such that $\cD_M$ has the desired properties. The sequence will satisfy $\codim \cD_{m+1}=\codim\cD_m+\codim U$ and $\cE(\cP(\cD_{m+1}))>\cE(\cP(\cD_m))+\epsilon^3q^{-\codim U}$.

Suppose that we have produced some $\cD_m$ that does not satisfy property (2) above. Then there is some $1\leq i\leq k$ and $\epsilon|\cD_m|$ choices of $W\in\mathcal \cD_m$, say $\cD'\subset\cD_m$, such that for each $W\in \cD'$, there exists $x\in W\setminus U$ such that $f_i|_{x+W\cap U}$ is not $\epsilon$-regular.

Fix $W\in\cD'$ with associated $x\in W\setminus U$. Since $f_i|_{x+W\cap U}$ is not $\epsilon$-regular, by part (3) of Proposition~\ref{thm:energy}, there exists some $W'\leq W\cap U$ of codimension 1 such that $\mathcal E(\cP(W'|x+W\cap U))>\mathcal E(\cP(W\cap U|x+W\cap U))+\epsilon^2$. Now we pick $U'\leq W'$ such that $\dim W-\dim(W\cap U)=\dim (W\cap U)-\dim U'$. Applying Example~\ref{thm:decomp-3} to $U'\leq (W\cap U)\leq W$ gives a decomposition $\cD_W$ of $W$ with respect to $W\cap U$ such that the partition $\cP(\cD_W)$ of $W\setminus U$ refines the partition $\cP(W'|W\setminus U)$ (the partition of $W\setminus U$ into cosets of $W'$). This implies that \[\cE(\cP(\cD_W))>\cE(\cP(W\cap U|W\setminus U))+\frac{\epsilon^2}{q^{\codim U}-1}.\]

Let $\cD_{m+1}$ be the decomposition of $V$ with respect to $U$ defined as follows. For each $W\in\cD'$, replace $W$ with $\cD_W$, the decomposition defined in the previous paragraph. For the remaining $W\in\cD_m\setminus\cD'$, replace $W$ with an arbitrary instance of Example~\ref{thm:decomp-3}. This refinement satisfies $\codim\cD_{m+1}=\codim\cD_m+\codim U$.

Now we compute the energy increment
\begin{align*}
\mathcal E(\cP(\cD_{m+1}))-\mathcal E(\cP(\cD_m))
&=\E_{W\in\cD_m}\left[\cE(\cP(\cD_{m+1})|_{W\setminus U})-\cE(\cP(\cD_m)|_{W\setminus U})\right]\\
&\geq\frac{|\cD'|}{|\cD_m|}\E_{W\in\cD'}\left[\cE(\cP(\cD_{m+1})|_{W\setminus U})-\cE(\cP(\cD_m)|_{W\setminus U})\right]\\
&\geq\epsilon^3 q^{-\codim U}.
\end{align*}

Since $0\leq\mathcal E(\cP(\cD))\leq k$ for any decomposition $\cD$, this process halts after $M\leq q^{\codim U}k\epsilon^{-3}$ iterations, giving the desired result.
\end{proof}

Given a decomposition $\cD$ of $V$ with respect to $U$, define \[V(\cD):=\bigcap_{W\in\cD}W.\] Note that this satisfies $\codim V(\cD)\leq| \cD|\cdot\codim\cD=\codim\cD\cdot q^{\codim\cD}$. Furthermore, note that if $\cD$ is non-trivial, i.e., $\codim\cD>0$, then $V(\cD)\leq U$. 

\begin{thm}[Strong decomposition regularity]
\label{thm:strong-weak-reg}
For a finite dimensional $\FF_q$-vector space $V$, functions $f_1,\ldots,f_k\colon V\to[-1,1]$, a subspace $V_0\leq V$, and a parameter $\epsilon>0$, there exists $n_{\mathsf{improved-reg}}=n_{\mathsf{improved-reg}}(\epsilon,q,k,\codim V_0)$ such that there exists a subspace $V_1\leq V_0$ and a decomposition $\cD$ of $V$ with respect to $V_1$ such that
\begin{enumerate}
\item $\codim V_1, \codim \cD\leq n_{\mathsf{improved-reg}}$;
\item $\cE(\cP(V_1|V\setminus V_1))\leq\cE(\cP(\cD))+\epsilon$;
\item for each $1\leq i\leq k$, for all but an $\epsilon$-fraction of $W\in\cD$, for all $x\in W\setminus V_1$ the function $f_i|_{x+W\cap V_1}$ is $\epsilon$-regular.
\end{enumerate}
\end{thm}

\begin{proof}
We can assume that $|V_0|\leq(\epsilon/4)|V|$.

Let $\cD_0$ be a decomposition of $V$ with respect to $V_0$ formed by applying Example~\ref{thm:decomp-3}. We iterate our weak arithmetic regularity lemma to produce a sequence of decompositions $\cD_0,\cD_1,\ldots,\cD_M$ such that $V(\cD_{M-1}),\cD_M$ has the desired properties. Let $V_m=V(\cD_m)$ and let $\cD_{m+1}$ be the result of applying Theorem~\ref{thm:weak-reg} to $V_m$ with parameter $\epsilon$. (Note that we must have $\dim V\geq \codim V_m\cdot q^{\codim V_m}\cdot k\cdot \epsilon^{-3}$ to apply Theorem~\ref{thm:weak-reg}. However, if this ever fails the desired result is trivially true.)

Note that $V_0\geq V_1\geq\cdots\geq V_M$ so $0\leq \mathcal E(\cP(V_0))\leq\mathcal E(\cP(V_1))\leq\cdots\leq\mathcal E(\cP(V_M))\leq k$. Therefore we can stop the iteration at $M\leq 2k\epsilon^{-1}$ such that $\mathcal E(\cP(V_M))\leq\mathcal E(\cP(V_{M-1}))+\epsilon/2$.

Now since $|V_{M-1}|\leq(\epsilon/4)|V|$, note that $\cE(\cP(V_M|V\setminus V_{M-1}))\leq\cE(\cP(V_{M-1}|V\setminus V_{M-1}))+\epsilon$. Furthermore $\cP(V_M|V\setminus V_{M-1})\succeq\cP(\cD_M)\succeq\cP(V_{M-1}|V\setminus V_{M-1})$. This implies property (2), completing the proof.
\end{proof}

\begin{proof}[Proof of Proposition~\ref{thm:regularity}]
We can assume that $|V_0|\leq(\epsilon/4)|V|$.

Apply Theorem~\ref{thm:strong-weak-reg} to $f_1,\ldots,f_k\colon V\to[-1,1]$ and $V_0\leq V$ with parameter $\min(\epsilon^3/4,1/(4k))$. Write $V_1\leq V_0$ for the subspace produced and $\cD$ for the decomposition produced. For a uniform random choice of $W\in\cD$, write $V_2=V_1\cap W$ and pick $U$ satisfying $U\oplus V_2=W$ arbitrarily. We claim that with positive probability, this choice of $U, V_2, V_1$ satisfies the desired properties.

First note that with probability at least $1-k(1/(4k))=3/4$ we have that for all $1\leq i\leq k$ and for all $x\in W\setminus V_1$ we have that $f_i|_{x+V_2}$ is $\epsilon$-regular.

Now by part (2) of Proposition~\ref{thm:energy}, we write
\begin{align*}
\epsilon^3/4
&\geq \cE(\cP(\cD))-\cE(\cP(V_1|V\setminus V_1))\\
&=\sum_{i=1}^k\norm{(f_i)_{\cP(V_1|V\setminus V_1)}-(f_i)_{\cP(\cD)}}_{L^2(V\setminus V_1)}^2\\
&=\E_{x\in V\setminus V_1}\left[\sum_{i=1}^k\left(\E_{y\in W(x)\cap V_1}[f_i(x+y)]-\E_{y\in V_1}[f_i(x+y)]\right)^2\right].
\end{align*}
In the last line above we write $W(x)\in\cD$ for the unique part of the decomposition containing $x$.

This implies that for a uniform random choice of $x\in V\setminus V_1$, with probability at least $1-\epsilon/4$, this $x$ satisfies
\begin{equation}
\label{eq:good-coset-decomp}
|\E_{y\in W(x)\cap V_1}[f_i(x+y)]-\E_{y\in V_1}[f_i(x+y)]|\leq \epsilon\quad\text{for all}\quad 1\leq i\leq k.\tag{$\S$}
\end{equation}
Therefore choosing $W\in\cD$ uniformly at random, the expected number of $x\in W$ failing to satisfy (\ref{eq:good-coset-decomp}) is at most $|W\cap V_1|+(\epsilon/4)(|W|-|W\cap V_1|)\leq(\epsilon/2)|W|$. By Markov's inequality, with probability at least $1/2$, at most an $\epsilon$-fraction of $x\in W$ fail to satisfy (\ref{eq:good-coset-decomp}). Thus by the union bound, with positive probability this choice of $U, V_2, V_1$ satisfies all the desired properties.
\end{proof}

\section{Extensions}
\label{sec:extensions}

\subsection{Infinite removal}
\label{ssec:infinite-removal}

First let us recall the infinite graph removal lemma \cite{AS08}.

\begin{thm}
For every (possibly infinite) set of graphs $\cH$ and $\epsilon>0$, there are $h_0$ and $\delta>0$ such that the following holds. If a graph $G$ has induced $H$-density at most $\delta$ for each $H\in\cH$ on at most $h_0$ vertices, then $G$ can be made induced $\cH$-free by changing at most an $\epsilon$-fraction of the edges.
\end{thm}

This result is proved by a short modification of the usual proof of induced graph removal. By performing an analogous modification to the proof of induced arithmetic removal it is possible to deduce the following infinite arithmetic removal lemma.

\begin{thm}
\label{thm:infinite}
Fix $\cH$ a (possibly infinite) set of $r$-colored complexity 1 patterns over $\FF_q$. For every $\epsilon>0$ there are $k_0$ and $\delta>0$ such that the following holds. Given a finite dimensional $\FF_q$-vector space $V$, if a coloring $\phi\colon V\to[r]$ has $H$-density at most $\delta$ for each $H\in\cH$ with at most $k_0$ columns, then $\phi$ can be made $\cH$-free by changing the color of at most an $\epsilon$-fraction of the elements.
\end{thm}

To prove this infinite arithmetic removal lemma we use the same recoloring algorithm as in the proof of Theorem~\ref{thm:main} with one modification. We replace the regularity recoloring result (Proposition~\ref{thm:regularity-recoloring}) with the following strong variant.

\begin{prop}[Strong variant of Proposition~\ref{thm:regularity-recoloring}]
\label{thm:regularity-recoloring-strong}
For a finite dimensional $\FF_q$-vector space $V$, a coloring $\phi\colon V\to[r]$, and parameters $0<\epsilon\leq1$ and $\bm\epsilon'=(\epsilon_0',\epsilon_1',\ldots)$ satisfying $\epsilon_0'\geq\epsilon_1'\geq\cdots>0$, there is some $n_{\mathsf{reg}}=n_{\mathsf{reg}}(\epsilon,\bm\epsilon',q,r)$ such that there exist subspaces $V_2\leq V_1\leq V$, a complement $U$ satisfying $U\oplus V_1=V$, and a recoloring $\phi'\colon V\to[r]$ that agrees with $\phi$ on all but at most $\epsilon|V|$ values satisfying:
\begin{enumerate}
\item $1/\epsilon \leq \codim V_1\leq \codim V_2\leq n_{\mathsf{reg}}$;
\item if a color appears in some coset $x+V_1$ under $\phi'$, then at least an $\epsilon/(2r)$-fraction of $x+V_2$ is that color under $\phi$;
\item for each $x\in U\setminus\{0\}$, the original coloring $\phi|_{x+V_2}$ is $\epsilon'_{\codim V_1}$-regular.
\end{enumerate}
\end{prop}

The only difference between Proposition~\ref{thm:regularity-recoloring} and Proposition~\ref{thm:regularity-recoloring-strong} is that the degree of regularity in part (3) is now allowed to depend on $\codim V_1$. The proof strategy given in Section~\ref{sec:regularity} also proves this strong variant with an appropriate choice of parameters.

Now suppose we perform our recoloring procedure on an $r$-coloring $\phi\colon V\to[r]$. The output is a recoloring $\phi'\colon V\setminus\{0\}\to[r]$. Suppose that along the way the application of regularity produces $V_2\leq V_1\leq V$ where $\codim V_1=d$. We now define a ``summary'' of the recoloring $\phi\rightsquigarrow\phi'$ as follows. Let $\Phi\colon\FF_q^d\setminus\{0\}\to2^{[r]}\setminus\{\emptyset\}$ be the map which sends a coset $x+V_1$ to the set of colors which appear under $\phi'$ in $x+V_1$. The recoloring $\phi'|_{V_1}$ is a canonical coloring and thus is determined by a function $\chi\colon\FF_q\setminus\{0\}\to[r]$. We will call $(\Phi,\chi)$ the summary of the recoloring procedure $\phi\rightsquigarrow\phi'$. Note that the number of possible summaries is bounded in terms of $d$ (and $q,r$).

\begin{defn}
Let $H=(A,\psi)$ be an $r$-colored pattern over $\FF_q$ where $A$ has dimensions $\ell\times k$. Given $\Phi\colon\FF_q^d\setminus\{0\}\to2^{[r]}\setminus\{\emptyset\}$ and $\chi\colon\FF_q\setminus\{0\}\to[r]$, say that $(\Phi,\chi)$ \emph{partially induces an $H$-instance} if there exists $\bx\in(\FF_q^d)^k$ such that $A\bx=\bm0$ and setting $I\subset[k]$ to be the set of $i\in[k]$ such that $x_i=0$ the following holds:
\begin{enumerate}
    \item $\psi(i)\in\Phi(x_i)$ for each $i\in[k]\setminus I$;
    \item for some $H'=(A',\psi|_I)$, a subpattern of $H$ restricted to the variables $I$, the $\chi$-canonical coloring of $\FF_q^n\setminus\{0\}$ contains an $H'$-instance for all large enough $n$.
\end{enumerate}
\end{defn}

This definition is chosen so that if $\phi'$ contains an $H$-instance, then the summary $(\Phi,\chi)$ partially induces an $H$-instance. A careful reading of the proof of the main theorem shows that it proves the following:
\begin{prop}
\label{thm:proof-summary}
    Let $H$ be an $r$-colored complexity 1 pattern and fix $\epsilon>0$. There exist $\epsilon_{\mathsf{strong}}=\epsilon_{\mathsf{strong}}(\epsilon, H)>0$ and $\delta=\delta(\epsilon, H)>0$ such that the following holds. Suppose the recoloring procedure of Proposition~\ref{thm:regularity-recoloring-strong} applied to $\phi$ with parameters $\epsilon$ and $\bm\epsilon'=(\epsilon_{\mathsf{strong}},\epsilon_{\mathsf{strong}},\ldots)$ produces $\phi'$. Let $(\Phi,\chi)$ be the summary of this recoloring. If $(\Phi,\chi)$ partially induces an $H$-instance, then the $H$-density in $\phi$ is more than $\delta$.
\end{prop}

Now we give the idea of the proof of Theorem~\ref{thm:infinite}. Fix $\cH$ a possibly infinite set of $r$-colored complexity 1 patterns over $\FF_q$. We define finite subsets $\cH_1\subseteq\cH_2\subseteq\cdots\subseteq\cH$  as follows. Fix some $d$. Consider the finite set of all possible pairs $\Phi\colon\FF_q^d\setminus\{0\}\to2^{[r]}\setminus\{\emptyset\}$ and $\chi\colon\FF_q\setminus\{0\}\to[r]$. For each $(\Phi,\chi)$, if there exists any $H\in\cH$ such that $(\Phi,\chi)$ contains an $H$-instance add one such $H$ to $\cH_d$. This ensures that $\cH_d$ is finite.

Given $\epsilon>0$, we set \[\epsilon_d'=\min_{H\in\cH_d}\epsilon_{\mathsf{strong}}(\epsilon, H)\quad\text{and}\quad\delta_d=\min_{H\in\cH_d}\delta(\epsilon, H).\] We apply Proposition~\ref{thm:regularity-recoloring-strong} with $\bm\epsilon'=(\epsilon_0',\epsilon_1',\ldots)$; all other parts of the recoloring produces remain the same. At the end we obtain a coloring that we claim is $\cH$-free. However, by the definitions we made in the previous paragraph, it suffices to check that it is $\cH_{\codim V_1}$-free. Taking $\delta=\delta_{n_{\mathsf{reg}}(\epsilon,\bm\epsilon',q,r)}$, Proposition~\ref{thm:proof-summary} implies that if the recoloring is not $\cH_{\codim V_1}$-free, then it has $H$-density more than $\delta_{\codim V_1}\geq\delta$ for some $H\in\cH_{\codim V_1}\subseteq\cH$. Finally taking $k_0$ to be the maximum number of columns in any of the patterns of $\cH_{n_{\mathsf{reg}}(\epsilon,\bm\epsilon',q,r)}$ completes the argument.

\subsection{Inhomogeneous patterns}
\label{ssec:inhomogeneous}

\begin{defn}
Let $H=(A,\psi)$ be an $r$-colored pattern over $\FF_q$ where $A$ is an $\ell\times k$ matrix. For a finite dimensional $\FF_q$-vector space $V$, a coloring $\phi\colon V\to [r]$, and an offset $\bm b\in V^\ell$, say that $\bx\in V^k$ is an \emph{$(H,\bm b)$-instance} if $A\bx=\bm b$ and $\phi(x_i)=\psi(i)$.
\end{defn}

\begin{thm} 
Fix a finite set $\cH$ of $r$-colored complexity 1 patterns over $\FF_q$. For every $\epsilon>0$ there exists $\delta=\delta(\epsilon,\cH)>0$ such that the following holds. Given a finite dimensional $\FF_q$-vector space $V$, if a coloring $\phi\colon V\to[r]$, and offsets $\{\bm b_H\}_{H\in\cH}$ (where for $H=(A,\psi)\in\cH$ where $A$ has $\ell_H$ rows we have $\bm b_H\in V^{\ell_H}$) satisfy that $\phi$ has $(H,\bm b_H)$-density at most $\delta$ for every $H\in\cH$, then there exists a recoloring $\phi'\colon V\setminus B\to[r]$ that is $(H,\bm b_H)$-free for all $H\in\cH$ and differs from $\phi$ on at most $\epsilon|V|$ elements. Here \[B=\spn_{H\in \cH,i\in[\ell_H]}(\bm b_H)_i.\]
\end{thm}

We prove this by reduction to the main theorem. We have a subspace $B\leq V$ where $B$ is of bounded size. Pick a complement $\tilde V$ satisfying $B\oplus \tilde V=V$. Now we convert $\phi\colon V\to [r]$ to $\tilde\phi\colon\tilde V\to[r]^B$ in the obvious manner: $\tilde\phi(x):=(\phi(x+u))_{u\in B}$. One can check that given a pattern $H=(A,\psi)$ where $A$ is an $\ell\times k$ matrix and an offset $\bm b\in B^\ell$ we can create a set of $N=|B|^{k-\rank A}r^{k(|B|-1)}$ patterns $H_1,\ldots,H_N$ so that every $(H,\bm b)$-instance in $\phi$ becomes an $H_i$-instance in $\tilde\phi$ for some $1\leq i\leq N$. Then applying the main theorem to this finite set of homogeneous patterns, we deduce the desired result.


\end{document}